\documentclass[11pt, reqno]{amsart}
\usepackage{amssymb, amstext, amscd, amsmath, amsthm}
\usepackage{color}
\usepackage{dsfont}
\usepackage{hyperref}

\usepackage{xy}
\xyoption{all}

\theoremstyle{plain}
\newtheorem{theorem}{Theorem}[section]
\newtheorem{thm}[theorem]{Theorem}
\newtheorem{cor}[theorem]{Corollary}
\newtheorem{prop}[theorem]{Proposition}
\newtheorem{lem}[theorem]{Lemma}
\newtheorem*{theorem*}{Theorem}
\theoremstyle{definition}
\newtheorem{rem}[theorem]{Remark}
\newtheorem{defn}[theorem]{Definition}

\newtheorem{eg}[theorem]{Example}
\newtheorem{egs}[theorem]{Examples}

\newtheorem{notation}[theorem]{Notation}

\newcommand{\bN}{{\mathbb{N}}}

\newcommand{\bT}{{\mathbb{T}}}

\newcommand{\bZ}{{\mathbb{Z}}}

  \newcommand{\B}{{\mathcal{B}}}

  \newcommand{\G}{{\mathcal{G}}}

  \newcommand{\K}{{\mathcal{K}}}
\renewcommand{\L}{{\mathcal{L}}}

  \newcommand{\N}{{\mathcal{N}}}
\renewcommand{\O}{{\mathcal{O}}}

  \newcommand{\Q}{{\mathcal{Q}}}
  
\renewcommand{\S}{{\mathcal{S}}}
  \newcommand{\T}{{\mathcal{T}}}

\newcommand{\fs}{{\mathfrak{s}}}
\newcommand{\ft}{{\mathfrak{t}}}

\newcommand{\Be}{{\mathbf{e}}}

\newcommand{\upchi}{{\raise.35ex\hbox{\ensuremath{\chi}}}}

\newcommand{\qforal}{\quad\text{for all}\quad}

\newcommand{\Aut}{\operatorname{Aut}}

\newcommand{\id}{{\operatorname{id}}}

\newcommand{\spn}{\operatorname{span}}

\newcommand{\ca}{\mathrm{C}^*}

\newcommand{\mt}{\varnothing}

\newcommand{\ol}{\overline}

\begin{document}
\title[Self-Similar $k$-Graph C*-algebras]{Self-Similar $k$-Graph C*-algebras}
\author[H. Li]{Hui Li}
\address{Hui Li,
Research Center for Operator Algebras and Shanghai Key Laboratory of Pure Mathematics and Mathematical Practice, Department of Mathematics, East China Normal University, 3663 Zhongshan North Road, Putuo District, Shanghai 200062, China}
\email{lihui8605@hotmail.com}
\author[D. Yang]{Dilian Yang$^*$}\thanks{${}^*$ Corresponding author.}
\address{Dilian Yang,
Department of Mathematics $\&$ Statistics, University of Windsor, Windsor, ON
N9B 3P4, CANADA}
\email{dyang@uwindsor.ca}

\thanks{The first author was partially supported by Research Center for Operator Algebras of East China Normal University and was partially supported by Science and Technology Commission of Shanghai Municipality (STCSM), grant No. 13dz2260400.}
\thanks{The second author was partially supported by an NSERC Discovery Grant 808235.}

\begin{abstract}
In this paper, we introduce a notion of a self-similar action of a group $G$ on a $k$-graph $\Lambda$, and associate it a universal C*-algebra $\O_{G,\Lambda}$.
We prove that $\O_{G,\Lambda}$ can be realized as the Cuntz-Pimsner algebra of a product system. If $G$ is amenable and the action is pseudo free, then $\O_{G,\Lambda}$ is shown to be isomorphic to a ``path-like" groupoid C*-algebra.
This facilitates studying the properties of $\O_{G,\Lambda}$.
We show that $\O_{G,\Lambda}$ is always nuclear and satisfies the Universal Coefficient Theorem;
we characterize the simplicity of $\O_{G,\Lambda}$ in terms of the underlying action; and we prove that, whenever $\O_{G,\Lambda}$ is simple, there is a dichotomy:
it is either stably finite or purely infinite, depending on whether $\Lambda$ has nonzero graph traces or not.
Our main results generalize the recent work of Exel and Pardo on self-similar graphs.
\end{abstract}

\subjclass[2010]{46L05}
\keywords{Self-similar; $k$-Graph; Groupoid, Groupoid C*-algebra;  Kirchberg algebra; Product system}

\maketitle

\section{Introduction}

Let $G$ be a discrete (countable) group, and $(E,r,s)$ be a finite directed graph with no sources. In \cite{EP17}, Exel-Pardo introduced a notion of a self-similar action of $G$
on $E$, which naturally generalizes the notion of self-similar groups (see, e.g., \cite{Nek05}).
The main object they were interested in is $\O_{G,E}$, which is a unital universal C*-algebra generated by a unitary representation of $G$ and a Cuntz-Krieger representation of $E$,
such that these two representations are compatible with the given self-similar action $(G,E)$. Surprisingly,
this construction embraces many important and interesting C*-algebras, such as
unital Katsura algebras \cite{Kat08, Kat082} and certain C*-algebras of self-similar groups constructed by Nekrashevych \cite{Nek09}.
To study $\O_{G,E}$, Exel-Pardo in \cite{EP17} associated the self-similar action $(G,E)$ an inverse semigroup $\S_{G,E}$.
It turns out that its tight groupoid $\G_{\text{tight}}(\S_{G,E})$ plays a vital role in the study of $\O_{G,E}$ in \cite{EP17}.
To be a little bit more precise, Exel-Pardo showed that
$\O_{G,E}$ is isomorphic to $\ca(\G_{\text{tight}}(\S_{G,E}))$. So studying $\O_{G,E}$ pinpoints $\G_{\text{tight}}(\S_{G,E})$.
Thus there is no surprise that the paper \cite{EP17} heavily depends on the machinery
on inverse semigroups developed in \cite{Exe08, EP16}. Among other results, in \cite{EP17} Exel-Pardo characterized the minimality and topological principality (resp.~local contractingness) of
$\G_{\text{tight}}(\S_{G,E})$,  and so the simplicity (resp.~pure infiniteness) of $\O_{G,E}$ when $G$ is amenable.

To give a more concrete description of $\G_{\text{tight}}(\S_{G,E})$, Exel-Pardo associated a given pseudo free self-similar action $(G,E)$ a
``path-like" groupoid $\G_{G,E}$. This is a subgroupoid of a tail equivalence with lag group $\check{G}\rtimes_\rho \bZ$,
where $\check{G}$ is the corona $G^\infty/G^{(\infty)}$ and $\rho$ is the ``right shift".
So the new feature here is that the lag is not an integer as graph C*-algebras, but an element of $\check{G}\rtimes_\rho \bZ$.
Then it is shown that $\G_{\text{tight}}(\S_{G,E})$ is isomorphic to $\G_{G,E}$.

In this paper, we introduce self-similar group actions on $k$-graphs, which are higher-dimensional analogues of self-similar group actions on directed graphs.
In higher-dimensional cases, in general the construction of the inverse semigroup $\S_{G,E}$ in \cite{EP17}
mentioned above does not apply, and so unlike \cite{EP17} one can not apply the machinery in \cite{Exe08, EP16}.
Our strategies are the following. As in \cite{EP17}, we associate a ``path-like" groupoid $\G_{G,\Lambda}$, called a self-similar path groupoid,
to a given self-similar action of a group $G$ on a $k$-graph $\Lambda$. This generalizes the construction of $\G_{G,E}$ in the directed graph case.
Then we make full use of $\G_{G,\Lambda}$ directly. To obtain that $\O_{G,\Lambda}$ is isomorphic to $\ca(\G_{G,\Lambda})$,
we first construct a product system $X_{G,\Lambda}$
over $\bN^k$. It turns out that we have the following relations: $\O_{G,\Lambda}\cong\O_{X_{G,\Lambda}}$,
$\N\O_{X_{G,\Lambda}}\twoheadrightarrow\O_{G,\Lambda}$, and $\O_{G,\Lambda}\twoheadrightarrow\ca(\G_{G,\Lambda})$.
Then invoking a deep result on the uniqueness on the Cuntz-Nica-Pimsner algebra
$\N\O_{X_{G,\Lambda}}$ of  Carlsen-Larsen-Sims-Vittadello in \cite{CLSV11},
we are able to prove that $\N\O_{X_{G,\Lambda}}$ is isomorphic to $\ca(\G_{G,\Lambda})$, and so
all four associated C*-algebras are isomorphic to each other:
$\O_{G,\Lambda}\cong\O_{X_{G,\Lambda}}\cong\mathcal{NO}_{X_{G,\Lambda}}\cong \ca(\G_{G,\Lambda})$.
Through $\ca(\G_{G,\Lambda})$, we show that if the self-similar action is pseudo free and $G$ is amenable, then
(i) $\O_{G,\Lambda}$ is always nuclear and satisfies the Universal Coefficient Theorem (UCT) \cite{RS87}; (ii) the simplicity
of $\O_{G,\Lambda}$ can be characterized in terms of the action itself and $\Lambda$; (iii) when $\O_{G,\Lambda}$
is simple, $\O_{G,\Lambda}$ has a dichotomy: it is either stably finite or purely infinite.
In particular, whenever $\O_{G,\Lambda}$ is simple and purely infinite, it is classifiable by its K-theory thanks to the Kirchberg-Phillips classification theorem \cite{Kir94, Phi00}.
We should mention that (ii) and (iii) are proved by applying some  recent important results on the C*-algebras of groupoids (e.g., \cite{BL17, BCFS14, RS17}).
On the way to our main results, we also obtain that $\G_{G,\Lambda}$ is amenable as well.

To summarize, our main results are the following:

\begin{theorem*}[Theorems \ref{T:OO} and \ref{T:ONC}]

Let $\Lambda$ be a $k$-graph such that $\Lambda^0$ is finite, and $G$ be a self-similar group on $\Lambda$.  Then we have the following:
\begin{itemize}
\item[(i)]
$\mathcal{O}_{G,\Lambda}\cong\mathcal{O}_{X_{G,\Lambda}}$.

\item[(ii)] If, furthermore, $G$ is amenable and the action is pseudo free, then
\[
\O_{G,\Lambda}\cong\O_{X_{G,\Lambda}}\cong\mathcal{NO}_{X_{G,\Lambda}}\cong \ca(\G_{G,\Lambda}).
\]
\end{itemize}

\end{theorem*}

\begin{theorem*}[Theorems \ref{T:AmeSim} and \ref{T:dicho}]
Let $G$ be an amenable group, and $\Lambda$ be a $k$-graph with $\Lambda^0$ finite.
Suppose that $G$ has a pseudo free self-similar action on $\Lambda$. Then
\begin{itemize}
\item[(i)]
$\mathcal{O}_{G,\Lambda}$ is nuclear and satisfies the UCT;
\item[(ii)] $\mathcal{O}_{G,\Lambda}$ is simple $\Leftrightarrow$ $\Lambda$ is $G$-cofinal and $G$-aperiodic.
\end{itemize}

 Suppose that $\mathcal{O}_{G,\Lambda}$ is simple. Then $\O_{G,\Lambda}$ is either stably finite or purely infinite.
 More precisely,
 \begin{itemize}
\item[(iii)]
$\mathcal{O}_{G,\Lambda}$ is stably finite $\Leftrightarrow$ $\Lambda$ has nonzero graph traces;

\item[(iv)]
$\mathcal{O}_{G,\Lambda}$ is purely infinite $\Leftrightarrow$ $\Lambda$ has no nonzero graph traces.
\end{itemize}
\end{theorem*}

The paper is structured as follows. In Section \ref{S:Pre}, some necessary backgrounds which will be used later are provided.
The notions of self-similar group actions on $k$-graphs and their associated C*-algebras are introduced in Section \ref{S:SS}.
Some basic properties are also given there. To a given self-similar action of a group $G$ on a $k$-graph $\Lambda$, in Section \ref{S:OC} we associate a product
system $X_{G,\Lambda}$ over $\bN^k$, whose Cuntz-Pimsner algebra is isomorphic to $\O_{G,\Lambda}$. This will be very useful when
we realize $\O_{G,\Lambda}$ as a groupoid C*-algebra in Section \ref{S:OG}. More precisely, in Section \ref{S:OG}, when a given self-similar
action $(G,\Lambda)$ is pseudo free, we construct a ``path-like groupoid" $\G_{G,\Lambda}$, which plays an indispensable role in studying $\O_{G,\Lambda}$.
The main result there is that $\O_{G,\Lambda}\cong\O_{X_{G,\Lambda}}\cong\mathcal{NO}_{X_{G,\Lambda}}\cong \ca(\G_{G,\Lambda})$ when
$G$ is amenable. Finally, in Section \ref{S:Pro}, with the aid of some recent important results on the C*-algebras of groupoids (see, e.g., \cite{BL17, BCFS14, RS17}),
the properties of $\O_{G,\Lambda}$ are studied in detail.

%


\subsection*{Notation and Conventions}
Let us end this section by fixing our notation and conventions.

Throughout this paper, $k$ is a fixed positive integer which is allowed to be $\infty$. For any semigroup $P$, let $P^k$ (resp.~$P^{(k)}$) denote
 the direct product (resp.~direct sum) of $k$ copies of $P$ (they coincide if $k < \infty$).
 Let $\mathbb{N}$ be the set of all nonnegative integers, and $\{e_i\}_{i=1}^{k}$ be the standard basis of $\mathbb{N}^k$.
 For $n,m \in \mathbb{N}^k$, we use $n \lor m$ (resp.~$n\wedge m$) to denote the coordinatewise maximum
(resp.~minimum) of $n$ and $m$. For $z \in \bT^k$ and $n=\sum_{i=1}^k n_i e_i$,
we use the multi-index notation $z^n:=\prod_{i=1}^{k}z_i^{n_i}$.

Finally, $G$ always stands for a discrete countable group, whose identity is written as $1_G$.

\section{Preliminaries}

\label{S:Pre}

In this section, we provide some necessary background which will be needed later.

\subsection{Groupoids}

In this subsection we recall the background of groupoid C*-algebras studied by Renault in \cite{Ren80}.

A \textit{groupoid} is a set $\Gamma$ such that there exist a subset $\Gamma^{(2)}$ of $\Gamma \times \Gamma$, a product map $\Gamma^{(2)} \to \Gamma, (\gamma_1, \gamma_2) \mapsto \gamma_1 \gamma_2$, and
an inverse map $\Gamma \to \Gamma, \gamma \mapsto \gamma^{-1}$ satisfying the following properties:
 $(\gamma^{-1})^{-1}=\gamma$ for all $\gamma \in \Gamma$;
$(\gamma,\gamma^{-1}),(\gamma^{-1},\gamma) \in \Gamma^{(2)}$ for all $\gamma \in \Gamma$;
if $(\gamma_1,\gamma_2), (\gamma_2,\gamma_3) \in \Gamma^{(2)}$, then $(\gamma_1 \gamma_2,\gamma_3),(\gamma_1,\gamma_2 \gamma_3) \in \Gamma^{(2)}$ and $(\gamma_1 \gamma_2) \gamma_3=\gamma_1(\gamma_2 \gamma_3)$;
 $(\gamma_1 \gamma_2)\gamma_2^{-1}=\gamma_1$ and $\gamma_1^{-1}(\gamma_1 \gamma_2)=\gamma_2$ for all $(\gamma_1,\gamma_2) \in \Gamma^{(2)}$.
Moreover, there are \emph{range} and \emph{source maps} $r,s :\Gamma \to \Gamma$ defined by $r(\gamma):=\gamma \gamma^{-1}$ and $s(\gamma):=\gamma^{-1} \gamma$, respectively.
Denote by $\Gamma^{(0)}:=s(\Gamma)=r(\Gamma)$ the \emph{unit space} of $\Gamma$. Furthermore, $\Gamma$ is called a \emph{topological groupoid} if $\Gamma$ is a topological space such that the product and inverse maps are both continuous.

Throughout this paper, \textsf{by locally compact groupoids, we mean second countable, locally compact, Hausdorff topological groupoids.}

Let $\Gamma$ be a locally compact groupoid.
$\Gamma$ is said to be \emph{\'{e}tale} if its range and source maps are both local homeomorphisms.
A subset $E \subseteq \Gamma$ is called a \emph{bisection} if the restrictions $r \vert_E$ and $s \vert_E$ are both homeomorphisms onto their images. An \'{e}tale groupoid is called \textit{ample} if it has a basis of compact open bisections.
Let $\Gamma$ be an \'{e}tale groupoid. Given $u \in \Gamma^{(0)}$, let $\Gamma_u:=s^{-1}(u)$ and $\Gamma^u:=r^{-1}(u)$. Then $\Gamma_u^u:=\Gamma_u \cap \Gamma^u$ is called the \emph{isotropy group} at $u$.
$\Gamma$ is said to be \emph{topologically principal} if the set of units whose isotropy groups are trivial is dense in $\Gamma^{(0)}$. A subset $U \subseteq \Gamma^{(0)}$ is said to be \emph{invariant} if for any $\gamma\in \Gamma$, $s(\gamma) \in U$ implies $r(\gamma) \in U$. $\Gamma$ is said to be \emph{minimal} if the only open invariant subsets of $\Gamma^{(0)}$ are $\mt$ and $\Gamma^{(0)}$.

Let $\Gamma$ be an \'{e}tale groupoid. For $f,g \in C_c(\Gamma)$, define
\[
\Vert f \Vert_I:=\sup_{u \in \Gamma^{(0)}}\Big\{\sum_{\gamma \in r^{-1}(u)}\vert f(\gamma) \vert, \sum_{\gamma \in s^{-1}(u)}\vert f(\gamma) \vert \Big\};
\]
\[
f\cdot g (\gamma):=\sum_{s(\beta)=r(\gamma)}f(\beta^{-1})g(\beta\gamma); \text{ and } f^*(\gamma):=\overline{f(\gamma^{-1})}.
\]
Then $\Vert\cdot\Vert_I$ is a norm, which is called the \emph{$I$-norm}, and $C_c(\Gamma)$ is a $*$-algebra. A $*$-representation $\pi$ of $C_c(\Gamma)$ on a Hilbert space is said to be \emph{bounded} if $\Vert \pi(f) \Vert \leq \Vert f \Vert_I$ for all $f \in C_c(\Gamma)$. Define $\Vert f \Vert:=\sup_\pi \Vert \pi(f) \Vert$, where $\pi$ runs through all bounded $*$-representations of $C_c(\Gamma)$. Then $\Vert\cdot\Vert$ is a C*-norm on $C_c(\Gamma)$, and the completion of $C_c(\Gamma)$ under the $\Vert\cdot\Vert$-norm is called the \emph{(full) groupoid C*-algebra}, denoted as $\ca(\Gamma)$.

\subsection{Type semigroups}

In this subsection, we briefly introduce type semigroups of ample groupoids from \cite{RS17}, which is a generalization of type semigroups arising from discrete group actions on compact Hausdorff spaces from \cite{RS12}
(also cf.~\cite{Wag93}).

Let $S$ be a unital abelian semigroup. One can associate $S$ a pre-order $\le$, which is called the \emph{algebraic ordering}, as follows: for $s_1,s_2 \in S$, $s_1 \leq s_2$ if there exists $t \in S$ such that $s_1+t=s_2$.
$S$ is said to be \emph{almost unperforated} if for $s_1,s_2 \in S$ and $n,m \in \mathbb{N}$, we have $ns_1 \leq ms_2, n >m$ imply $s_1 \leq s_2$;
and \emph{purely infinite} if $2s \leq s$ for all $s \in S$.

It is straightforward to see that purely infinite $\Rightarrow$ almost unperforated.

\begin{defn}
Let $\Gamma$ be an ample groupoid. Define an equivalence $\sim_\Gamma$ on $C_c(\Gamma^{(0)},\mathbb{N})$ as follows: for $f,g \in C_c(\Gamma^{(0)},\mathbb{N})$,
\[
f \sim_\Gamma g \quad\text{if}\quad  f=\sum_{i=1}^{n}\mathds{1}_{s(E_i)}\text{ and } g=\sum_{i=1}^{n}\mathds{1}_{r(E_i)}
\]
for some compact open bisections $E_1,\dots,E_n$ of $\Gamma$.
The \emph{type semigroup} $S(\Gamma)$ of $\Gamma$ is defined to be $S(\Gamma):=C_c(\Gamma^{(0)},\mathbb{N})/\sim_\Gamma$. The equivalence class of $f \in C_c(\Gamma^{(0)},\mathbb{N})$ in $S(\Gamma)$ is written as $[f]$.
\end{defn}

The following theorem was proved independently by B\"{o}nicke-Li in \cite{BL17} and by Rainone-Sims in \cite{RS17}.

\begin{thm}
\label{T:BLRS}
Let $\Gamma$ be a minimal, topologically principal, amenable, ample groupoid with the compact unit space. Then $\ca(\Gamma)$ is stably finite if and only if there exists a faithful tracial state on $\ca(\Gamma)$. Furthermore, suppose that $S(\Gamma)$ is almost unperforated. Then $\ca(\Gamma)$ is either stably finite or purely infinite.
\end{thm}

\subsection{$k$-Graph C*-algebras}

In this subsection, we recap the background of $k$-graph C*-algebras from \cite{KP00}.

A \textit{$k$-graph} is a countable small category such that there exists a functor $d:\Lambda \to \mathbb{N}^k$ satisfying the \emph{factorization property}: for $\mu\in\Lambda, n,m \in \mathbb{N}^k$ with $d(\mu)=n+m$, there exist unique $\alpha,\beta \in \Lambda$ such that $d(\alpha)=n, d(\beta)=m, s(\alpha)=r(\beta)$, and $\mu=\alpha\beta$. Let $(\Lambda_1,d_1)$ and $(\Lambda_2,d_2)$ be two $k$-graphs. A functor $f:\Lambda_1 \to \Lambda_2$ is called a \emph{graph morphism} if $d_2 \circ f=d_1$.

Let $\Lambda$ be a $k$-graph. For $\mu \in \Lambda$, let $\vert\mu\vert:=\sum_{i=1}^{k}d(\mu)_i$. For $n\in \bN^k$, let $\Lambda^n:=d^{-1}(n)$. For $A,B \subseteq \Lambda$, define $AB:=\{\mu\nu:\mu\in A,\nu\in B, s(\mu)=r(\nu)\}$. For $\mu,\nu \in \Lambda$, define $\Lambda^{\min}(\mu,\nu):=\{(\alpha,\beta) \in \Lambda \times \Lambda:\mu\alpha=\nu\beta,d(\mu\alpha)=d(\mu)\lor d(\nu)\}$.

Let $\Lambda$ be a $k$-graph. Then $\Lambda$ is said to be \emph{row-finite} (resp.~\textit{without sources}) if $\vert v\Lambda^{n}\vert<\infty$
(resp.~  $v\Lambda^{n} \neq \mt$ ) for all $v \in \Lambda^0$ and $n \in \mathbb{N}^k$.

Throughout this paper, \textsf{all $k$-graphs are assumed to be row-finite and without sources.}

\begin{eg}
Define $\Omega_k:=\{(p,q) \in \mathbb{N}^k \times \mathbb{N}^k:p \leq q\}$. For $(p,q), (q,m) \in \Omega_k$, define $(p,q) \cdot (q,m):=(p,m)$, $r(p,q):=(p,p)$, $s(p,q):=(q,q)$, and $d(p,q):=q-p$. Then $(\Omega_k,d)$ is a row-finite $k$-graph without sources.
\end{eg}

\begin{defn}
Let $\Lambda$ be a $k$-graph. Then a family of partial isometries $\{S_\mu\}_{\mu \in \Lambda}$ in a C*-algebra $B$ is called a \emph{Cuntz-Krieger $\Lambda$-family} if
\begin{enumerate}
\item\label{S_v ort proj} $\{S_v\}_{v \in \Lambda^0}$ is a family of mutually orthogonal projections;
\item\label{S_munu=S_muS_nu} $S_{\mu\nu}=S_{\mu} S_{\nu}$ if $s(\mu)=r(\nu)$;
\item\label{S_mu*S_mu=S_s(mu)} $S_{\mu}^* S_{\mu}=S_{s(\mu)}$ for all $\mu \in \Lambda$; and
\item\label{CK-condition for all path row finite} $S_v=\sum_{\mu \in v \Lambda^{n}}S_\mu S_\mu^*$ for all $v \in \Lambda^0, n \in \mathbb{N}^k$.
\end{enumerate}
The C*-algebra $\mathcal{O}_\Lambda$ generated by a universal Cuntz-Krieger $\Lambda$-family $\{s_\mu\}_{\mu \in \Lambda}$ is called the \emph{$k$-graph C*-algebra} of $\Lambda$.
\end{defn}

Let $\Lambda$ be a $k$-graph and let $\{S_\mu\}_{\mu \in \Lambda}$ be a Cuntz-Krieger $\Lambda$-family. A strongly continuous homomorphism $\alpha: \bT^k \to \Aut(\ca(S_\mu))$ is called a \emph{gauge action} if $\alpha_z(S_\mu)=z^{d(\mu)}S_\mu$ for all $z \in \bT^k$ and $\mu \in \Lambda$. By the universal property of $\mathcal{O}_\Lambda$, there exists a gauge action $\gamma$ for $\{s_\mu\}_{\mu \in \Lambda}$.

The following theorem is the gauge-invariant uniqueness theorem for $k$-graph $\ca$-algebras.

\begin{thm}
\label{gauge-inv uni thm k-graph}
Let $\Lambda$ be a $k$-graph and let $\{S_\mu\}_{\mu \in \Lambda}$ be a Cuntz-Krieger $\Lambda$-family which admits a gauge action. Then the homomorphism $h:\mathcal{O}_\Lambda \to \ca(S_\mu)$ induced from the universal property of $\mathcal{O}_\Lambda$ is an isomorphism if and only if $S_v \neq 0$ for all $v \in \Lambda^0$.
\end{thm}

Let $\Lambda$ be a $k$-graph. Then a graph morphism from $\Omega_k$ to $\Lambda$ is called an \emph{infinite path} of $\Lambda$. Let $\Lambda^\infty$
denote the set of all infinite paths of $\Lambda$.
For $\mu \in \Lambda$,  define the cylinder $Z(\mu):=\{\mu x: x \in \Lambda^\infty,s(\mu)=x(0,0)\}$. Endow $\Lambda^\infty$ with the topology generated by the basic open sets $\{Z(\mu):\mu \in \Lambda\}$. Each $Z(\mu)$ is compact, so $\Lambda^\infty$ is second countable, locally compact, Hausdorff, and totally disconnected. $\Lambda^\infty$ is compact if and only if $\Lambda^0$ is finite. Finally define the \emph{path groupoid} of $\Lambda$ by
\begin{align*}
\mathcal{G}_\Lambda:=\{(\mu x,d(\mu)-d(\nu),\nu x) \in \Lambda^\infty \times \mathbb{Z}^k \times \Lambda^\infty: s(\mu)=s(\nu),x \in s(\mu)\Lambda^\infty\}.
\end{align*}
For $\mu,\nu \in \Lambda$ with $s(\mu)=s(\nu)$, let $Z(\mu,\nu):=\{(\mu x,d(\mu)-d(\nu),\nu x):x \in s(\mu)\Lambda^\infty\}$. Endow $\mathcal{G}_\Lambda$ with the topology generated by the basic open sets $Z(\mu,\nu)$.
Then $\mathcal{G}_\Lambda$ is an ample groupoid and each $Z(\mu,\nu)$ is a compact open bisection.

Hjelmborg in \cite{Hje01} introduced the notion of graph traces for directed graphs (see also \cite{Tom03}). Pask-Rennie-Sims generalized this notion to $k$-graphs
in \cite{PRS08} as follows. Let $\Lambda$ be a $k$-graph. A function $gt:\Lambda^0 \to [0,\infty)$
is called a \emph{graph trace} if $gt(v)=\sum_{\mu \in v \Lambda^p}gt(s(\mu))$ for all $v \in \Lambda^0,p \in \mathbb{N}^k$. The graph trace $gt$ is said to be \emph{faithful} if $gt(v) \neq 0$ for all $v \in \Lambda^0$.


\subsection{Product systems over $\mathbb{N}^k$}
In this subsection we recap some background on product systems over $\mathbb{N}^k$ from \cite{CLSV11, Fow02, SY10}.

Let $A$ be a C*-algebra. A \emph{C*-correspondence} over $A$ (see \cite{Fow02, FMR03}) is a right Hilbert $A$-module $X$ together with a $*$-homomorphism $\phi: A\to \L(X)$, which gives a left
action of $A$ on $X$ by $a\cdot x:=\phi(a)(x)$ for all $a\in A$ and $x\in X$. $X$ is said to be \textit{essential} if $\ol\spn\{\phi(a) (x): a\in A, x\in X\}=X$.

Let $A$ be a C*-algebra. A \emph{product system} over $\mathbb{N}^k$ with coefficient $A$ is a semigroup $X=\bigsqcup_{n \in \mathbb{N}^k}X_n$, where each
 $X_n$ is an essential C*-correspondences over $A$ with $\phi_n(A)\subseteq \K(X_n)$, such that the following properties hold: $X_0=A$;
the multiplication $X_0 \cdot X_n$ (resp.~$X_n \cdot X_0$) is implemented by the left (resp.~right) action of $A$ on $X_n$ for all $n \in \mathbb{N}^k$;
$X_n \cdot X_m \subseteq X_{n+m}$ for all $n$, $m \in \mathbb{N}^k$;
there is an isomorphism $\Phi_{n,m}$ from $X_n \otimes_A X_m$ onto $X_{n+m}$ for $n,m \in \mathbb{N}^k$,
where $X_n \otimes_A X_m$ denotes the balanced tensor product, by sending $x\otimes y$ to $xy$ for all $x \in X_n$ and $y \in X_m$.

Let $A, B$ be C*-algebras, let $X$ be a product system over $\mathbb{N}^k$ with coefficient $A$, and let $\psi:X \to B$ be a map.
For $n \in \mathbb{N}^k$, denote by $\psi_n:=\psi \vert_{X_n}$. Then $\psi$ is called a \emph{representation} of $X$ if $\psi_n$ is a linear map for all $n \in \mathbb{N}^k; \psi_0$ is a homomorphism; $\psi_n(a\cdot x)=\psi_0(a)\psi_n(x)$ for all $n \in \mathbb{N}^k, a \in A,x \in X_n;\psi_n(x)^*\psi_n(y)=\psi_0(\langle x,y\rangle_A)$ for all $n \in \mathbb{N}^k,x,y \in X_n$; and $\psi_n(x)\psi_m(y)=\psi_{n+m}(xy)$ for all $n, m \in \mathbb{N}^k, x \in X_n, y \in X_m$. In addition, if $\psi$ is a representation, then for $n \in \mathbb{N}^k$ there exists a unique homomorphism $\psi_n^{(1)}:\mathcal{K}(X_n) \to B$ such that $\psi_n^{(1)}(\Theta_{x,y})=\psi_n(x)\psi_n(y)^*$ for all $x,y \in X_n$.

\begin{notation}
Let $A$ be a C*-algebra, and $X$ be a product system over $\mathbb{N}^k$ with coefficient $A$. For $a \in A$, denote by $L_a \in \mathcal{K}(X_0)$ such that $L_a(b):=ab$ for all $b \in A$. For $p,q \in \mathbb{N}^k$, when $p \neq 0$, define a homomorphism $\iota_{p}^{q}:\mathcal{L}(X_p) \to \mathcal{L}(X_q)$ by
\begin{align*}
\iota_{p}^{q}(T):= \begin{cases}
    \Phi_{p,q-p} \circ (T \otimes \id ) \circ \Phi_{p,q-p}^{-1}  &\text{ if $p \leq q$} \\
    0 &\text{ otherwise }.
\end{cases}
\end{align*}
When $p=0$, define a homomorphism $\iota_{p}^{q}:\mathcal{K}(X_p) \to \mathcal{L}(X_q)$ by $\iota_{p}^{q}(L_a):=\phi_q(a)$.

For $p \in \mathbb{N}^k$, define a closed two-sided ideal of $A$ by
\begin{align*}
I_p:= \begin{cases}
    \bigcap_{0 \neq q \leq p}\ker(\phi_q)  &\text{ if $p \neq 0$} \\
    A &\text{ otherwise }.
\end{cases}
\end{align*}
We then obtain a C*-correspondence over $A$ by $\widetilde{X_p}:=\bigoplus_{q \leq p} X_q \cdot I_{p-q}$.
For $p,q \in \mathbb{N}^k$, define the following homomorphisms.
\begin{align*}
\widetilde{\iota}_{p}^{q}:&\mathcal{L}(X_p) \to \mathcal{L}(\widetilde{X_q}), \quad \widetilde{\iota}_{p}^{q}(T):=\bigoplus_{r \leq q}\iota_{p}^{r}(T)\vert_{X_r \cdot I_{q-r}} \quad \text{if } p \neq 0, \\
 \widetilde{\iota}_{p}^{q}:&\mathcal{K}(X_p) \to \mathcal{L}(\widetilde{X_q}), \quad  \widetilde{\iota}_{p}^{q}(T):=\bigoplus_{r \leq q}\iota_{p}^{r}(T)\vert_{X_r \cdot I_{q-r}} \quad\text{if } p=0.
\end{align*}
\end{notation}

\begin{defn}\label{covariant rep}
Let $A, B$ be C*-algebras, $X$ be a product system over $\mathbb{N}^k$ with coefficient $A$, and $\psi:X \to B$ be a representation.

1. $\psi$ is said to be \emph{Cuntz-Pimsner (or CP) covariant} if
\[
\psi_0(a)=\psi_p^{(1)}(\phi_p(a))\qforal p \in \mathbb{N}^k, a \in A.
\]

2. $\psi$ is said to be \emph{Cuntz-Nica-Pimsner (or CNP) covariant} if
\begin{enumerate}
\item\label{Nica cov} $\psi_p^{(1)}(T)\psi_q^{(1)}(S)=\psi_{p \lor q}^{(1)}(\iota_{p}^{p \lor q}(T) \iota_{q}^{p \lor q}(S))$ for all $p,q \in \mathbb{N}^k, T\in \mathcal{K}(X_p)$, $S \in \mathcal{K}(X_q)$;
\item\label{new CP cov} for any finite set $\{T_{p_i}: p_i \in \mathbb{N}^k,T_{p_i} \in \mathcal{K}(X_{p_i}) \}_{i=1}^{n}$, if there exists $p \in \mathbb{N}^k$ satisfying $\sum_{i=1}^{n}\widetilde{\iota}_{p_i}^{q}(T_{p_i})=0$ for all $q \geq p$, we have $\sum_{i=1}^{n}\psi_{p_i}^{(1)}(T_{p_i})=0$.
\end{enumerate}
\end{defn}

The \textit{Cuntz-Pimsner algebra $\O_X$} (resp.~\textit{Cuntz-Nica-Pimsner algebra $\N\O_X$}) is the C*-algebra generated by a universal CP (resp.~CNP) covariant representation
$j_{cp}$ (resp.~$j_{cnp}$) of $X$ (\cite[Proposition~2.9]{Fow02} and {\cite[Proposition~3.12]{SY10}}).

%

\begin{prop}\label{Fowler CP implies CNP}
Let $X$ be a product system over $\mathbb{N}^k$ with coefficient $A$. If $\psi:X \to B$ is a CP covariant representation, then $\psi$ is also CNP covariant.
\end{prop}
\begin{proof}
First of all, the condition in Definition~\ref{covariant rep} 2.(i) follows from \cite[Proposition~5.4]{Fow02}.

To see that the condition in Definition~\ref{covariant rep} 2.(ii) is also satisfied, let us fix a finite set $\{T_{p_i}: p_i \in \mathbb{N}^k,T_{p_i} \in \mathcal{K}(X_{p_i}) \}_{i=1}^{n}$ such that there exists $p \in \mathbb{N}^k$ satisfying $\sum_{i=1}^{n}\widetilde{\iota}_{p_i}^{q}(T_{p_i})=0$ for all $q \geq p$. We may assume that $p_1=0,T_{p_1}=L_a$ for some $a \in A$, and $p_2,\dots,p_n \neq 0$. Then there exists $q \geq p \lor p_1 \lor \cdots \lor p_n$ such that $\sum_{i=1}^{n}\iota_{p_i}^{q}(T_{p_i})=0$ (as the $q$-th summand). So $\phi_q(a)+\sum_{i=2}^{n}\iota_{p_i}^{q}(T_{p_i})=0$. Notice that $\phi_q(a) \in \mathcal{K}(X_{q})$ (which is required in our definition).
By \cite[Proposition~4.7]{Lan95}, $\iota_{p_i}^{q}(T_{p_i}) \in \mathcal{K}(X_{q})$ for all $2 \leq i \leq n$. Since $\psi$ is CFP covariant,  similar to the proof of \cite[Lemma~3.10]{Pim97}, we get
\begin{align*}
\psi_q^{(1)}(\phi_q(a))+\sum_{i=2}^{n}\psi_q^{(1)}(\iota_{p_i}^{q}(T_{p_i}))=\sum_{i=1}^{n}\psi_{p_i}^{(1)}(T_{p_i})=0.
\end{align*}
Therefore $\psi$ is CNP covariant.
\end{proof}

Let $X$ be a product system over $\mathbb{N}^k$ with coefficient $A$, and let $\psi$ be a representation of $X$. A \emph{gauge action} is a strongly continuous homomorphism $\alpha:\bT^k \to \mathrm{Aut}(\ca(\psi(X)))$ such that $\alpha_z(\psi_n(x))=z^n \psi_n(x)$ for all $z \in \bT^k, n \in \mathbb{N}^k, x \in X_n$.


\begin{thm}[{\cite[Corollary~4.12]{CLSV11}}]\label{gauge-inv uni thm}
Let $X$ be a product system over $\mathbb{N}^k$ with coefficient $A$, and $\psi$ be a CNP covariant representation of $X$ which admits a gauge action.
Denote by $h:\mathcal{NO}_X \to \ca(\psi(X))$ the homomorphism induced from the universal property of $\mathcal{NO}_X$. Suppose that $h \vert_{j_{cnp,0}(A)}$ is injective. Then $h$ is an isomorphism.
\end{thm}

\section{Self-Similar $k$-Graphs}

\label{S:SS}

In this section, we introduce self-similar group actions on $k$-graphs and their associated C*-algebras.
The work of this section was inspired by Exel-Pardo \cite{EP17}.

\begin{defn}
Let $\Lambda$ be a $k$-graph. A bijection $\varphi:\Lambda \to \Lambda$ is called an \emph{automorphism} of $\Lambda$ if
\begin{enumerate}
\item $\varphi(\Lambda^n) \subseteq \Lambda^n$ for all $n \in \mathbb{N}^k$;
\item $s \circ \varphi=\varphi \circ s$ and $r \circ \varphi=\varphi \circ r$.
\end{enumerate}
Denote by $\Aut(\Lambda)$ the group of all automorphisms of $\Lambda$.

Let $G$ be a (discrete countable) group. We say that \textit{$G$ acts on $\Lambda$} if there is a group homomorphism $\varphi$ from $G$ to $\Aut(\Lambda)$.
For $g\in G$ and $\mu\in\Lambda$, we often simply write $\varphi(g)(\mu)$ as $g\cdot \mu$.
\end{defn}

We should mention that the above notions are different from those in \cite{KP00}.


\begin{defn}\label{D:sskg}
Let $\Lambda$ be a $k$-graph, and $G$ be a group acting on $\Lambda$.
Then the action is said to be \emph{self-similar} if there exists a \emph{restriction map} $G\times \Lambda \to G$,
$(g,\mu)\mapsto g|_\mu$, such that
\begin{enumerate}
\item
$g\cdot (\mu\nu)=(g \cdot \mu)(g \vert_\mu \cdot \nu)$ for all $g \in G,\mu,\nu \in \Lambda$ with $s(\mu)=r(\nu)$.

\item
$g \vert_v =g$ for all $g \in G,v \in \Lambda^0$;

\item
$g \vert_{\mu\nu}=g \vert_\mu \vert_\nu$ for all $g \in G,\mu,\nu \in \Lambda$ with $s(\mu)=r(\nu)$;

\item
$1_G \vert_{\mu}=1_G$ for all $\mu \in \Lambda$;

\item
$(gh)\vert_\mu=g \vert_{h \cdot \mu} h \vert_\mu$ for all $g,h \in G,\mu \in \Lambda$.
\end{enumerate}
Moreover, $\Lambda$ and $G$ are called a \textit{self-similar $k$-graph} over $G$ and \textit{self-similar group on $\Lambda$}, respectively.
\end{defn}

\begin{defn}
Let $\Lambda$ be a self-similar $k$-graph over a group $G$, and $H$ be a subset of $\Lambda^0$.
Then $H$ is said to be \emph{$G$-hereditary} if $s(H \Lambda) \subseteq H$ and $G \cdot H \subseteq H$.
\end{defn}

Notice that $H$ is hereditary in the usual sense, and that it naturally induces a self-similar action of $G$ on $H \Lambda$.

\begin{defn}
Let $\Lambda$ be a self-similar $k$-graph over a group $G$.
Then $\Lambda$ is said to be \emph{$G$-strongly connected} if $(G \cdot v) \Lambda w \neq \mt$ for all $v,w \in \Lambda^0$.
\end{defn}

Obviously, if $\Lambda$ is strongly connected, then it is $G$-strongly connected.

\begin{lem}
Let $\Lambda$ be a self-similar $k$-graph over a group $G$.
Then
\begin{enumerate}
\item\label{g vert_mu cdot s(nu)} $g \vert_\mu \cdot s(\nu)=g \cdot s(\nu)$ for all $g \in G, \mu,\nu \in \Lambda$ with $s(\mu)=r(\nu)$;
\item $g|_\mu\cdot s(\mu)=g\cdot s(\mu)$ for all $\mu \in \Lambda$;
\item $g \vert_\mu^{-1}=g^{-1} \vert_{g \cdot \mu}$ for all $g \in G,\mu \in \Lambda$.
\end{enumerate}
\end{lem}
\begin{proof}
(i) This can be seen from the following:
\[
g\cdot s(\nu)=g\cdot s(\mu\nu)=s(g\cdot (\mu\nu))=s(g|_\mu \cdot \nu)=g|_\mu\cdot s(\nu).
\]

(ii) This is a special case of (i).

(iii) This follows from
$
(g|_\mu)(g^{-1}|_{g\cdot \mu})=(gg^{-1})|_{g\cdot\mu}=1_G=(g^{-1}g)|_{\mu}=(g^{-1}|_{g\cdot \mu})(g|_\mu ).
$
\end{proof}

Let $\Lambda$ be a $k$-graph. Put $\Lambda^\Be:=\bigcup_{i=1}^{k}\Lambda^{e_i}$. Then the following lemma shows how one can extend an action of $G$
on $\Lambda^0 \cup \Lambda^\Be$ with a restriction to a self-similar action of $G$ on $\Lambda$.

\begin{lem}\label{extend self similar}
Let $\Lambda$ be a $k$-graph, and $G$ be a group. Suppose that $G$ acts on the set $\Lambda^0 \cup \Lambda^\Be$,
and that there is a restriction map $G \times \big(\Lambda^0 \cup \Lambda^\Be \big) \to G, (g,x)\mapsto g|_x$, satisfying the following properties:
\begin{enumerate}
\item $G \cdot \Lambda^n \subseteq \Lambda^n$ for all $n \in \{0,e_i:1 \leq i \leq k\}$;
\item $s(g \cdot \mu)=g \cdot s(\mu)$ and $r(g \cdot \mu)=g \cdot r(\mu)$ for all $g \in G, \mu \in\Lambda^\Be$;
\item $g \vert_\mu \cdot s(\nu)=g \cdot s(\nu)$ for all $g \in G,\mu \in \Lambda^\Be,\nu \in \Lambda$ with $s(\mu)=r(\nu)$;
\item $(g \cdot \mu)( g \vert_\mu \cdot \nu)=(g \cdot \alpha)(g \vert_\alpha \cdot \beta)$ for all $g \in G, \mu,\nu,\alpha,\beta \in \Lambda^\Be$ with $\mu\nu=\alpha\beta$;
\item $g \vert_v =g$ for all $g \in G,v \in \Lambda^0$;
\item
$g \vert_\mu \vert_\nu=g \vert_\alpha \vert_\beta$ for all $g \in G,\mu,\nu,\alpha,\beta \in\Lambda^\Be$ with $\mu\nu=\alpha\beta$;
\item $(gh)\vert_\mu=(g \vert_{h \cdot \mu})( h \vert_\mu)$ for all $g,h \in G,\mu \in \Lambda^\Be$.
\end{enumerate}
Then there exists a unique self-similar action of $G$ on $\Lambda$ with the restriction map $\vert:G \times \Lambda \to G$ extending the given action and the given map $\vert$.
\end{lem}

\begin{proof}
For $\mu \in \Lambda$ with $\vert\mu\vert=2$, write $\mu=\mu_1 \mu_2$ with $\mu_1,\mu_2 \in \Lambda^\Be$.
For $g\in G$, put $g \cdot \mu:=(g \cdot \mu_1)(g \vert_{\mu_1} \cdot \mu_2)$ and $g|_\mu:=g|_{\mu_1}|_{\mu_2}$.
It is not hard to see that both $g \cdot \mu$ and $g|_\mu$ are well-defined. Inductively, we extend the given action and restriction to $\Lambda$.
It is easy to check that they satisfy Definition~\ref{D:sskg} (i)-(v) are satisfied. Therefore the extensions yield a self-similar action of $G$ on $\Lambda$.
\end{proof}

Let $\Lambda$ be a self-similar $k$-graph over a group $G$. For $g \in G$ and $x \in \Lambda^\infty$, we define
\begin{align*}
(g \cdot x)(p,q)&:=g \vert_{x(0,p)} \cdot x(p,q)\qforal(p,q) \in \Omega_k,\\
g \vert_x(p)&:=g \vert_{x(0,p)}\qforal p\in \bN^k.
\end{align*}
Then $g\cdot x\in \Lambda^\infty$ and $g \vert_x$ is a function from $\mathbb{N}^k$ to $G$. The following lemma gives some basic properties of $g\cdot x$ and $g \vert_x$.

\begin{lem}
\label{L:5.2}
Let $\Lambda$ be a self-similar $k$-graph over a group $G$. Then
\begin{itemize}
\item[(i)] for $g \in G$, the map $\Lambda^\infty \to \Lambda^\infty,x \mapsto g \cdot x$ is a homeomorphism;
\item[(ii)] the map $G \times \Lambda^\infty \to \Lambda^\infty, (g,x) \mapsto g \cdot x$ induces a group action of $G$ on $\Lambda^\infty$;
\item[(iii)] $\sigma^{p}(g \cdot x)=g \vert_{x(0,p)} \cdot \sigma^p(x)$ for all $g \in G$, $p \in \mathbb{N}^k$ and $x \in \Lambda^\infty$;
\item[(iv)] $g \cdot (\mu x)=(g \cdot \mu)(g \vert_\mu \cdot x)$ for all $g \in G$, $\mu \in \Lambda$ and $x \in s(\mu)\Lambda^\infty$.
\end{itemize}
\end{lem}
\begin{proof}
The proofs of (i) and (iv) are straightforward. In what follows, we only prove (ii) and (iii).

(ii) Let $g_1,g_2\in G$, $x\in \Lambda^\infty$ and $(p,q)\in \Omega_k$. Using Definition \ref{D:sskg} (1), we have
\begin{align*}
(g_1\cdot(g_2\cdot x))(p,q)
&=g_1|_{g_2\cdot x(p,q)}((g_2\cdot x)(p,q))\\
&=g_1|_{g_2\cdot x(p,q)}(g_2|_{x(0,p)}\cdot x(p,q))\\
&=(g_1g_2)|_{x(0,p)}\cdot x(p,q)\\
&=((g_1g_2)\cdot x)(p,q).
\end{align*}

(iii) For $p\in \bN^k$ and $(s,t)\in\Omega_k$, repeatedly using Definition \ref{D:sskg} (5) gives
\begin{align*}
\sigma^p(g\cdot x)(s,t)
&=(g\cdot x)(s+p,t+p)\\
&=g|_{x(0,s+p)}\cdot x(s+p,t+p)\\
&=g|_{x(0,s+p)}\cdot \sigma^p(x)(s,t)\\
&=g|_{x(0,p)}|_{\sigma^p(x)(0,s)}\cdot \sigma^p(x)(s,t)\\
&=(g|_{x(0,p)}\cdot{\sigma^p(x))}(s,t).
\end{align*}
We are done.
\end{proof}

We now associate a self-similar $k$-graph a universal C*-algebra.

\begin{defn}\label{define O_G,Lambda}
Let $\Lambda$ with $\vert\Lambda^0\vert<\infty$ be a self-similar $k$-graph over a group $G$. Define the \textit{self-similar $k$-graph C*-algebra} $\mathcal{O}_{G,\Lambda}$ to be the universal unital C*-algebra generated by a family of unitaries $\{u_g\}_{g \in G}$ and
a Cuntz-Krieger family $\{s_\mu\}_{\mu \in \Lambda}$ satisfying
\begin{itemize}

\item[(i)]
$u_{gh}=u_g u_h$ for all $g$ and $h \in G$;
\item[(ii)]
$u_g s_\mu=s_{g \cdot \mu} u_{g \vert_\mu}$ for all $g \in G$ and $\mu \in \Lambda$.
\end{itemize}
\end{defn}

In other words, the self-similar $k$-graph C*-algebra $\O_{G,\Lambda}$ is generated by a `universal' pair $(u,s)$ of representations, where $u$ is a unitary representation of $G$ and
$s$ is a representation of the $k$-graph C*-algebra $\O_\Lambda$, such that $u$ and $s$
are compatible with the given self-similar action (i.e., (ii) holds).  The remark below shows that $\O_{G,\Lambda}$ exists nontrivially: $s_\mu \neq 0$ and $u_g \neq 0$ for all $\mu \in \Lambda$ and $g \in G$.

\begin{rem}
\label{R:nontrivial}
For $\mu \in \Lambda$ and $g \in G$, define
\begin{align*}
S_\mu(\delta_x):=\begin{cases}
   \delta_{\mu x}   &\text{ if $s(\mu)=x(0,0)$} \\
   0   &\text{ otherwise },
\end{cases}
\hskip 1cm  U_g(\delta_x):=\delta_{g \cdot x}.
\end{align*}
By Lemma \ref{L:5.2}, one can check that $\{S_\mu\}_{\mu \in \Lambda}$ is a Cuntz-Krieger $\Lambda$-family in $\B(\ell^2(\Lambda^\infty))$ and $\{U_g\}_{g \in G}$ is a family of unitaries in $\B(\ell^2(\Lambda^\infty))$ satisfying Definition~\ref{define O_G,Lambda} (i)-(ii). Observe that $S_v\ne 0$ and $U_g \neq 0$ for all $v \in \Lambda^0,g \in G$. So $s_\mu \neq 0$ and $u_g \neq 0$ for all $\mu \in \Lambda,g \in G$.
\end{rem}

In the sequel, we give some examples of $\O_{G,\Lambda}$.

\begin{egs}
Let $\Lambda$ with $\vert\Lambda^0\vert<\infty$ be a self-similar $k$-graph over a group $G$.
\begin{enumerate}
\item[1.]
Suppose that $G$ is trivial. Then $\O_{G,\Lambda}\cong \O_\Lambda$.

\item[2.]
Suppose that $g \vert_\mu=g$ for all $g\in G$ and $\mu\in \Lambda$. Let $\{t_\mu\}_{\mu \in \Lambda}$ be the Cuntz-Krieger $\Lambda$-family of $\mathcal{O}_\Lambda$. Then there exists a homomorphism $\alpha:G \to \Aut(\mathcal{O}_\Lambda)$ such that $\Phi(g)(t_\mu)=t_{g \cdot \mu}$ for all $g \in G,\mu \in \Lambda$. It follows easily that $\O_{G,\Lambda}\cong \O_\Lambda\rtimes_{\alpha} G$.

\item[3.]
Suppose that $\Lambda$ is a 1-graph. Then $\O_{G,\Lambda}$ is the C*-algebra studied by Exel-Pardo in \cite{EP17}, which includes Katsura algebras (\cite{Kat082}) and C*-algebras of self-similar groups constructed by Nekrashevych (\cite{Nek04, Nek09}).

\item[4.]
Suppose that $\Lambda$ is a single-vertex $k$-graph. For $1 \leq i \leq k$, let $n_i:=\vert\Lambda^{e_i}\vert$, and $\Lambda^{e_i}=\{x_{\fs}^i\}_{\fs =0}^{n_i-1}$. Suppose that for any $\mu \in \Lambda,h \in G$, there exists $g \in G$ such that $g \vert_\mu=h$. Then $\O_{G,\Lambda}$ is isomorphic to the boundary quotient C*-algebra $\Q(\Lambda\bowtie G)$ due to \cite[Theorem~3.3]{LY17}. In particular, when $G=\mathbb{Z}$ and
\begin{align*}
1\cdot {x_\fs^i}&=x_{(\fs+1)\ \text{mod } n_i}^i\text{ if }0\le \fs\le n_i-1\\
1|_{x_\fs^i}&=\left\{
\begin{matrix}
0 &\text{ if }0\le \fs< n_i-1\\
1 &\text{ if }\fs=n_i-1
\end{matrix}\right.\\
x^i_\fs x^j_\ft = x^j_{\ft'} x^i_{\fs'} &\text{ if } 1\leq i<j \leq k, \fs+\ft n_i=\ft'+\fs' n_j,
\end{align*}
we studied $\O_{\mathbb{Z},\Lambda}$ intensively in \cite{LY17}.
\end{enumerate}

The last two examples actually provide the main motivation of this paper.
\end{egs}

The following is a C*-algebra version of Lemma \ref{extend self similar}, whose easy proof is omitted.

\begin{lem}
Let $\Lambda$ be a self-similar $k$-graph over a group $G$ with $\vert\Lambda^0\vert<\infty$. Suppose that $\{U_g\}_{g \in G}$ is a family of unitaries
and $\{S_\mu:\mu \in \Lambda^0 \cup \Lambda^\Be\}$ is a family of partial isometries in a unital C*-algebra $B$ satisfying
\begin{enumerate}
\item $\{S_v\}_{v \in \Lambda^0}$ is a family of mutually orthogonal projections;
\item $\sum_{v \in \Lambda^0}S_v=1_B$;
\item $S_{\mu}^* S_{\mu}=S_{s(\mu)}$ for all $\mu \in \Lambda^\Be$;
\item $S_{\mu} S_{\nu}=S_\alpha S_\beta$ if $\mu,\nu,\alpha,\beta \in \Lambda^\Be$ with $\mu\nu=\alpha\beta$;
\item $S_v=\sum_{\mu \in v \Lambda^n}S_\mu S_\mu^*$ for all $v \in \Lambda^0$, $n\in\{e_i:1 \leq i \leq k\}$;
\item $U_{gh}=U_g U_h$ for all $g,h \in G$;
\item $U_g S_\mu=S_{g \cdot \mu} U_{g \vert_\mu}$ for all $g \in G,\mu \in \Lambda^0 \cup \Lambda^\Be$.
\end{enumerate}
Then there exists a unique Cuntz-Krieger $\Lambda$-family $\{T_\mu:\mu \in \Lambda\}$ in $B$ such that $T_\mu=S_\mu$ for all $\mu \in \Lambda^0 \cup \Lambda^\Be$ and $U_g T_\mu=T_{g \cdot \mu} U_{g \vert_\mu}$ for all $g \in G,\mu \in \Lambda$.
\end{lem}

\begin{prop}\label{generator of O_G, Lambda}
Let $\Lambda$ be a self-similar $k$-graph over a group $G$ with $\vert\Lambda^0\vert<\infty$. Then $\spn\{s_\mu u_g s_\nu^*: \mu, \nu \in \Lambda, g \in G, s(\mu)= g \cdot s(\nu)\}$ is a dense $*$-subalgebra of $\mathcal{O}_{G,\Lambda}$.
\end{prop}

\begin{proof}
First notice that $s(\mu)\ne g\cdot s(\nu)$ then $s_\mu u_g s_\nu^*=0$ as
\[
(s_\mu u_g s_\nu^*)^*s_\mu u_g s_\nu^*=s_\nu u_{g^{-1}}s_\mu^*s_\mu u_g s_\nu^*
 =s_\nu u_{g^{-1}} s_{s(\mu)} u_g s_\nu^*=s_\nu  s_{g^{-1}\cdot s_{s(\mu)}}  s_\nu^*.
\]
It follows from Definition~\ref{define O_G,Lambda}  that the given linear span is closed under the multiplication.
Also, for $\mu \in \Lambda$ and $g\in G$,  we have $s_\mu=s_\mu u_{1_G} s_{s(\mu)}^*$ and $u_g=\sum_{v\in \Lambda^0}s_{g\cdot v} u_g s_v$.
Therefore this linear span is dense in $\mathcal{O}_{G,\Lambda}$.
\end{proof}


\section{Realizing $\O_{G,\Lambda}$ as a Cuntz-Pimsner Algebra}

\label{S:OC}

Let $\Lambda$ be a self-similar $k$-graph over a group $G$ with $\vert\Lambda^0\vert<\infty$. In this section, we shall construct a product system $X_{G,\Lambda}$ over $\bN^k$
such that its Cuntz-Pimsner algebra $\mathcal{O}_{X_{G,\Lambda}}$ is isomorphic to $\O_{G,\Lambda}$.

Suppose that a group $G$ acts on a $k$-graph $\Lambda$ with $\vert\Lambda^0\vert<\infty$ self-similarly. Then the action restricts to an action of $G$ on $\Lambda^0$, say
$\varphi$. Let $C(\Lambda^0)$ be the set of all complex-valued functions on $\Lambda^0$. Then we obtain a C*-dynamical system $(C(\Lambda^0),G, {\varphi})$ such that ${\varphi}(g)(\delta_v)=\delta_{g \cdot v}$ for all $g \in G,v \in \Lambda^0$. Denote by $A_{G,\Lambda}:=C(\Lambda^0) \rtimes_{{\varphi}} G$, and let $\{j(\delta_v),j(g)\}_{v \in \Lambda^0,g \in G}$ be the generator set of $A_{G,\Lambda}$ (see \cite[Theorem~2.61]{Wil07}).

Recall that $A_{G,\Lambda}$ is a right Hilbert $A_{G,\Lambda}$-module (\cite{Lan95}).
For each $\mu \in \Lambda$, define a closed $A_{G,\Lambda}$-submodule of $A_{G,\Lambda}$ by
\[
A_{G,\Lambda}^{\mu}:=j(\delta_{s(\mu)})A_{G,\Lambda}=\overline{\mathrm{span}}\{j(\delta_{s(\mu)})j(g):g \in G\}.
\]
Then, for $p \in \mathbb{N}^k$, we form a right Hilbert $A_{G,\Lambda}$-module as follows:
\begin{align*}
X_{G,\Lambda,p} := \begin{cases}
    {\displaystyle \bigoplus_{\mu \in \Lambda^p}}A_{G,\Lambda}^{\mu} &\text{ if $p \neq 0$} \\
    A_{G,\Lambda} &\text{ if $p=0$ }.
\end{cases}
\end{align*}
Notice that two distinct paths $\mu, \nu\in\Lambda^p$ with $s(\mu)=s(\nu)$ produce two copies in $X_{G,\Lambda,p}$. To avoid confusion, let us from now on write $\chi_\mu \in X_{G,\Lambda,p}$ by
\begin{align*}
\chi_\mu(\nu):= \begin{cases}
    j(\delta_{s(\mu)})  &\text{ if $\mu=\nu$} \\
    0 &\text{ otherwise }.
\end{cases}
\end{align*}


One nice property of $X_{G,\Lambda,p}$ is given below.

\begin{lem}
\label{L:K=L}
$\mathcal{L}(X_{G,\Lambda,p})=\mathcal{K}(X_{G,\Lambda,p})$ for every $p\in\bN^k$.
\end{lem}

\begin{proof}
This is clearly true if $p=0$ as $A_{G,\Lambda}$ is unital.
But for $p \neq 0$,  one has
$1_{\mathcal{L}(X_{G,\Lambda,p})}=\sum_{\mu \in \Lambda^p}\Theta_{\chi_\mu, \chi_\mu}$
since
\[
\sum_{\mu \in \Lambda^p}\Theta_{\chi_\mu, \chi_\mu}(\chi_\nu a_\nu)
=\sum_{\mu \in \Lambda^p}\chi_\mu\langle \chi_\mu, \chi_\nu a_\nu\rangle
=\sum_{\mu \in \Lambda^p}\chi_\mu \chi_\mu^* \chi_\nu a_\nu
=\chi_\nu a_\nu
\]
for all $\nu\in\Lambda^p$.
\end{proof}

Now we endow each $X_{G,\Lambda,p}$ with the C*-correspondence structure over $A_{G,\Lambda}$.
If $p=0$, then $X_{G,\Lambda,0}=A_{G,\Lambda}$, and so it is a C*-correspondence over $A_{G,\Lambda}$. For $p \neq 0$ and for $v \in \Lambda^0$, define $\pi_{G,\Lambda,p}(j(\delta_v)) \in \mathcal{L}(X_{G,\Lambda,p})$ to be the projection onto the closed $A_{G,\Lambda}$-submodule $\bigoplus_{\mu \in v\Lambda^p}A_{G,\Lambda}^{\mu}$. For $g \in G$, define $U_{G,\Lambda,p}(g) \in U(\mathcal{L}(X_{G,\Lambda,p}))$ by
\[
U_{G,\Lambda,p}(g)\big(\chi_\mu a_\mu\big)=  \chi_{g\cdot \mu} j(g|_\mu) a_\mu\qforal \mu \in \Lambda^p, \ a_\mu \in A_{G,\Lambda}^\mu.
\]
It is easy to check that $(\pi_{G,\Lambda,p},U_{G,\Lambda,p})$ is a covariant homomorphism of the C*-dynamical system $(C(\Lambda^0),G,{\varphi})$. That is,
$U_{G,\Lambda,p}(g)\pi_{G,\Lambda,p}(\delta_v)=\pi_{G,\Lambda,p}( \delta_{g\cdot v})U_{G,\Lambda,p}(g)$ for all $g\in G$ and $v\in \Lambda^0$.
Then there exists an essential homomorphism $\phi_{G,\Lambda,p}:A_{G,\Lambda} \to \mathcal{L}(X_{G,\Lambda,p})= \mathcal{K}(X_{G,\Lambda,p})$. So this gives $X_{G,\Lambda,p}$ a
left $A_{G,\Lambda}$-module structure. Therefore, $X_{G,\Lambda,p}$ is a C*-correspondence over $A_{G,\Lambda}$.

The identity in the following lemma turns out to be very useful.

\begin{lem}
\label{L:self-similar}
For $0 \neq p \in \mathbb{N}^k$, we have
\[
j(g) \cdot \chi_\mu = \chi_{g\cdot \mu} j(g|_\mu)\qforal g\in G\text{ and } \mu\in \Lambda^p.
\]
\end{lem}
\begin{proof}
It follows directly from the definition of $U_{G,\Lambda,p}$.
\end{proof}

Define
\[
X_{G,\Lambda}:=\amalg_{p \in \mathbb{N}^k}X_{G,\Lambda,p}.
\]
For $p \in \mathbb{N}^k$, the multiplications $X_{G,\Lambda,0} \cdot X_{G,\Lambda,p}$ and $X_{G,\Lambda,p} \cdot X_{G,\Lambda,0}$ are defined to be the left and right actions of $A_{G,\Lambda}$ on $X_{G,\Lambda,p}$,
respectively.
For nonzero $p, q \in \mathbb{N}^k$, for $\mu \in \Lambda^p, \nu \in \Lambda^q$, and for $g,h \in G$, the multiplication is given as follows:
\begin{align*}
(\chi_\mu j(g)) \cdot (\chi_\nu j(h))
 := \begin{cases}
    \chi_{\mu (g \cdot \nu)} j(g|_\nu h) &\text{ if } s(\mu)=r(g \cdot \nu), \\
    0 &\text{ otherwise }.
\end{cases}
\end{align*}
Under the above multiplication, one can check that $X_{G,\Lambda}$ is a product system over $\mathbb{N}^k$ with coefficient $A_{G,\Lambda}$.

Let $j_{cnp}:X_{G,\Lambda} \to \mathcal{NO}_{X_{G,\Lambda}}$ be the universal CNP
covariant representation of $X_{G,\Lambda}$. It follows easily that $\{j_{cnp}(j(\delta_v)),j_{cnp}(j(g)),j_{cnp}(\chi_\mu):v \in \Lambda^0,g \in G,\mu \in \Lambda \setminus \Lambda^0\}$ generates $\mathcal{NO}_{X_{G,\Lambda}}$.

\begin{prop}\label{P:NOO}
Keep the above notation. Then
there exists a surjective homomorphism $h$ from $\mathcal{NO}_{X_{G,\Lambda}}$ onto $\mathcal{O}_{G,\Lambda}$.
\end{prop}

\begin{proof}
By the universal property of $A_{G,\Lambda}$, there exists a homomorphism $\psi_0:A_{G,\Lambda} \to \mathcal{O}_{G,\Lambda}$ such that $\psi_0(j(\delta_v))=s_v, \psi_0(j(g))=u_g$ for all $v \in \Lambda^0$ and $g \in G$. For $0 \neq p \in \mathbb{N}^k$, define a map $\psi_p:X_{G,\Lambda,p} \to \mathcal{O}_{G,\Lambda}$ by
\[
\psi_p\left(\sum_{\mu \in \Lambda^p} \chi_\mu a_\mu\right):=\sum_{\mu\in\Lambda^p} s_\mu \psi_0(a_\mu).
\]
By piecing $\{\psi_p\}_{p \in \mathbb{N}^k}$ together we obtain a representation $\psi$ of $X_{G,\Lambda}$ into $\mathcal{O}_{G,\Lambda}$.

For $a \in A_{G,\Lambda}$, we compute that
\begin{align*}
\psi_p^{(1)}(\phi_{G,\Lambda,p}(a))
&=\psi_p^{(1)}\left(\phi_{G,\Lambda,p}(a)\sum_{\mu\in\Lambda^p}\Theta_{\chi_\mu,\chi_\mu}\right)\\
&=\psi_p^{(1)}\left(\sum_{\mu\in\Lambda^p}\Theta_{\phi_{G,\Lambda,p}(a)\chi_\mu,\chi_\mu}\right)\\
&=\sum_{\mu\in\Lambda^p}\psi_p(\phi_{G,\Lambda,p}(a)\chi_\mu)\psi_{p}(\chi_\mu)^*\\
&=\sum_{\mu\in\Lambda^p}\psi_{0}(a)s_\mu s_\mu^*\\
&=\psi_0(a).
\end{align*}
So $\psi$ is CP covariant. By Proposition~\ref{Fowler CP implies CNP}, $\psi$ is CNP covariant. 
Hence there exists a unique homomorphism $h: \mathcal{NO}_{X_{G,\Lambda}} \to \mathcal{O}_{G,\Lambda}$ such that $h \circ j_{cnp}=\psi$. Clearly $h$ is also surjective.
\end{proof}

The above result will be strengthened when the action is required to be pseudo free (see Definition \ref{D:pf} and Theorem \ref{T:ONC}).
However, without pseudo freeness, it turns out that $\mathcal{O}_{G,\Lambda}$ is isomorphic to the Cuntz-Pimsner algebra $\mathcal{O}_{X_{G,\Lambda}}$ of $X_{G,\Lambda}$.

\begin{thm}
\label{T:OO}
Keep the above notation. Then
$\mathcal{O}_{X_{G,\Lambda}}\cong\mathcal{O}_{G,\Lambda}$.
\end{thm}

\begin{proof}
Let $\psi$ be the CP covariant representation of $X_{G,\Lambda}$ constructed in Proposition~\ref{P:NOO}.
Then there exists a unique homomorphism $h':\mathcal{O}_{X_{G,\Lambda}} \to \mathcal{O}_{G,\Lambda}$ such that $h' \circ j_{cp}=\psi$.

Conversely, define
\begin{align*}
S_\mu&:=
j_{G,\Lambda,d(\mu)}(\chi_\mu) \text{ for all }\mu\in\Lambda,
\\
 U_g&:=j_{G,\Lambda,0}(j(g))\text{ for all }g\in G.
\end{align*}
One can verify that $\{S_\mu\}_{\mu \in \Lambda}$ is a Cuntz-Krieger family, that $\{U_g\}_{g \in G}$ is a family of unitaries, and that they satisfy
Definition~\ref{define O_G,Lambda} (i)-(ii). By the universal property of $\mathcal{O}_{G,\Lambda}$, there exists a unique homomorphism $\rho:\mathcal{O}_{G,\Lambda} \to \mathcal{O}_{X_{G,\Lambda}}$ such that $\rho(s_\mu)=S_\mu,\rho(u_g)=U_g$ for all $\mu \in \Lambda,g \in G$. It is straightforward to see that $h' \circ \rho=\id, \rho \circ h'=\id$. Therefore $\mathcal{O}_{G,\Lambda}$ and $\mathcal{O}_{X_{G,\Lambda}}$ are isomorphic.
\end{proof}

\section{Realizing $\O_{G,\Lambda}$ as a Groupoid C*-algebra}

\label{S:OG}

Let $\Lambda$ be a self-similar $k$-graph over a group $G$.  In this section, we first associate it a ``path-like" groupoid $\G_{G,\Lambda}$. To make $\G_{G,\Lambda}$
a (locally compact) ample groupoid, we have to require that the self-similar action be pseudo free (see Definition~\ref{D:pf}).
Our main result here is that  $\O_{G,\Lambda}\cong \O_{X_{G,\Lambda}}\cong \N\O_{X_{G,\Lambda}}\cong\ca(\G_{G,\Lambda})$ if $\vert\Lambda^0\vert<\infty$ and $G$ is amenable.
This will play a very important role in next section.

\smallskip

Let $G$ be a group. Denote by $C(\mathbb{N}^k,G)$ the set of all maps from $\mathbb{N}^k$ to $G$, which is a group under the pointwise multiplication. For $f,g \in C(\mathbb{N}^k,G)$, define $f \sim g$ if there exists $z \in \mathbb{Z}^k$ such that $f(p)=g(p)$ for all $p \geq z$. Then $\sim$ is an equivalence relation on $C(\mathbb{N}^k,G)$. Denote by $Q(\bN^k,G)$ the quotient group $C(\mathbb{N}^k,G)/\sim$. As usual, if $f\in C(\mathbb{N}^k,G)$, then we write $[f] \in Q(\mathbb{N}^k,G)$.

For $z \in \mathbb{Z}^k$, $f \in C(\mathbb{N}^k,G)$ and $p \in \mathbb{N}^k$, define a translation
\begin{align*}
\T_z(f)(p)&:= \begin{cases}
    f(p-z) &\text{if $p-z \in \mathbb{N}^k$} \\
    1_G &\text{if $p-z \notin \mathbb{N}^k$}.
\end{cases}
\end{align*}
Then $\T_z$ yields an automorphism, still denoted by $\T_z$, on $Q(\bN^k,G)$.
Moreover, $\T:\mathbb{Z}^k \to \Aut Q(\bN^k,G),\ z\mapsto \T_z$, is a homomorphism.
So one can form the semidirect product $Q(\bN^k,G) \rtimes_\T \mathbb{Z}^k$.

\subsection{Self-similar path groupoids $\G_{G,\Lambda}$}

Let $\Lambda$ be a self-similar $k$-graph over a group $G$. For $g \in G$ and $x \in \Lambda^\infty$, recall from Section \ref{S:SS}
that  $g\cdot x\in \Lambda^\infty$ and $g \vert_x\in C(\mathbb{N}^k,G)$.
The following identity will be frequently used later.
%
%
%
\begin{lem}
\label{L:translation}
Let $\Lambda$ be a self-similar $k$-graph over a group $G$. Then
\[
[g|_{\mu x}]=\T_{d(\mu)}([g|_\mu|_x])\qforal g\in G,\ \mu \in \Lambda, \ x\in s(\mu)\Lambda^\infty.
\]
\end{lem}

\begin{proof}
 For $p\in \bN^k$ big enough (say $p\ge d(\mu)$), one has $g|_{\mu x}(p)=(g|_\mu)|_x(p-d(\mu))=\T_{d(\mu)}((g|_\mu)|_x)(p)$, as desired.
\end{proof}

\begin{defn}
Let $\Lambda$ be a self-similar $k$-graph over a group $G$. The \emph{self-similar path groupoid}, which is a subgroupoid of $\Lambda^\infty \times (Q(\bN^k,G)\rtimes_\T \mathbb{Z}^k ) \times \Lambda^\infty$, is defined to be
\[
\G_{G,\Lambda}:=\left\{(x;[f],p-q;y):
\begin{matrix} \sigma^p(x)=f(p) \cdot \sigma^q(y),\\
f(p+n)=f(p)\vert_{\sigma^q(y)}(n)\ \forall n \in \mathbb{N}^k
\end{matrix}
\right\}.
\]
\end{defn}
One can check that
\[
\G_{G,\Lambda}:=\left\{\big(\mu (g \cdot x);\T_{d(\mu)}([g \vert_{x}]),d(\mu)-d(\nu);\nu x\big):
\begin{matrix}
g \in G, \mu,\nu \in \Lambda ,\\
\text{ with } s(\mu)=g \cdot s(\nu)
\end{matrix}
\right\}.
\]
This description will  actually be very useful later.

For $g \in G, \mu,\nu \in \Lambda$ with $s(\mu)=g \cdot s(\nu)$, define
\[
Z(\mu,g,\nu):=\big\{(\mu (g \cdot x);\T_{d(\mu)}([g \vert_{x}]),d(\mu)-d(\nu);\nu x):x \in s(\nu)\Lambda^\infty\big\}.
\]
Endow $\mathcal{G}_{G,\Lambda}$ with the topology generated by the basic open sets
\[
\mathcal{B}_{G,\Lambda}:=\big\{Z(\mu,g,\nu):g \in G, \mu,\nu \in \Lambda,s(\mu)=g \cdot s(\nu)\big\}.
\]

In order to make $\G_{G,\Lambda}$ a nice locally compact groupoid, we have to introduce the following condition.

\begin{defn}
\label{D:pf}
Let $\Lambda$ be a self-similar $k$-graph over a group $G$.
Then the self-similar action is said to be \emph{pseudo free} if for any $g\in G,\mu\in\Lambda$, we have $g \cdot \mu=\mu$ and $g \vert_\mu=1_G \implies g=1_G$.
In this case, $\Lambda$ is called a \textit{pseudo free self-similar $k$-graph over $G$}.
\end{defn}

\begin{lem}
Let $\Lambda$ be a self-similar $k$-graph over a group $G$. Suppose that for any $g\in G$ and $\mu\in\Lambda^\Be$, we have $g \cdot \mu=\mu$ and
$g \vert_\mu=1_G \implies g=1_G$. Then the action is pseudo free.
\end{lem}
\begin{proof}
It is straightforward to prove the lemma by the induction on $\vert\mu\vert$.
\end{proof}

\begin{lem}
\label{L:5.7}
Let $\Lambda$ be a pseudo free self-similar $k$-graph over a group $G$. Then for any $g,\tilde{g} \in G, \mu \in \Lambda$, we have $g\cdot \mu=\tilde g\cdot \mu$ and
$g|_\mu=\tilde g|_{\mu} \implies \tilde g=g$.
\end{lem}

\begin{proof}
This follows from
\begin{align*}
(\tilde g^{-1}g)|_\mu=\tilde g^{-1}|_{g\cdot \mu}g|_\mu=\tilde g^{-1}|_{\tilde g\cdot \mu}\tilde g|_\mu=(\tilde g^{-1}\tilde g)|_\mu=1_G|_\mu=1_G.
\end{align*}
We are done.
\end{proof}

The following corollary will be useful later.

\begin{cor}
\label{C:5.7}
Let $\Lambda$ be a pseudo free self-similar $k$-graph over a group $G$. Then for any $g,\tilde{g} \in G, x \in \Lambda^\infty$, we have
$g\cdot x=\tilde g\cdot x$ and $[g|_x]=[\tilde g|_x] \implies \tilde g=g$.
\end{cor}

\begin{proof}
Notice that, for $p\in\bN^k$ big enough, one has
\[
g|_{x(0,p)}=\tilde g|_{x(0,p)} \text{ and } g\cdot {x(0,p)}=\tilde g\cdot {x(0,p)}.
\]
Then Lemma \ref{L:5.7} applies.
\end{proof}

Using the terminology in \cite{EP17}, then the above corollary says that pseudo freeness implies the following property: every $1_G\ne g\in G$ has at most finitely many minimal strongly fixed paths.


\begin{thm}\label{psdo free implies Hau}
Let $\Lambda$ be a pseudo free self-similar $k$-graph over a group $G$. Then $\mathcal{G}_{G,\Lambda}$ is an ample groupoid, and $\mathcal{B}_{G,\Lambda}$ consists of compact open bisections.
\end{thm}

\begin{proof}
We shall only prove that $\mathcal{B}_{G,\Lambda}$ is a base of $\G_{G,\Lambda}$, and that
$\G_{G,\Lambda}$ is Hausdorff. The proofs of the other properties are easy and omitted.

To see that $\mathcal{B}_{G,\Lambda}$ is a base of $\G_{G,\Lambda}$, let
\[
(\mu g \cdot  x;\T_{d(\mu)}([g \vert_{x}]),d(\mu)-d(\nu);\nu x)\in Z(\mu,g,\nu) \cap Z(\mu',g',\nu').
\]
Then there is some $x'\in \Lambda^\infty$ such that
\begin{align*}
&(\mu g \cdot x;\T_{d(\mu)}([g \vert_{x}]),d(\mu)-d(\nu);\nu x)\\
=&
(\mu' g' \cdot  x';\T_{d(\mu')}([g' \vert_{x'}]),d(\mu')-d(\nu');\nu' x')\in Z(\mu',g',\nu').
\end{align*}
So from $\nu x=\nu'x'$ one has $x=\alpha y$ and $x'=\alpha' y$ with $y\in \Lambda^\infty$ and
$(\alpha, \alpha')\in \Lambda^{\min}(\nu,\nu')$.
Combining this with $\mu g \cdot x=\mu' g' \cdot  x'$ gives $\mu g\cdot (\alpha y)=\mu' g'\cdot (\alpha' y)$. Thus
$\mu g\cdot \alpha (g_{\alpha}\cdot y)=\mu' g'\cdot \alpha' (g|_{\alpha'}\cdot y)$.
In particular, $\mu g\cdot \alpha=\mu' g'\cdot \alpha'$. This implies $(g\cdot \alpha, g' \cdot \alpha') \in \Lambda^{\min} (\mu,\mu')$
and
\begin{align*}
g|_{\alpha}\cdot y=g|_{\alpha'}\cdot y.
\end{align*}
Also, it follows from $\T_{d(\mu)}([g \vert_{x}])=\T_{d(\mu')}([g' \vert_{x'}])$ that
$\T_{d(\mu)}([g \vert_{\alpha y}])=\T_{d(\mu')}([g' \vert_{\alpha' y}])$.
By Lemma \ref{L:translation}, one has that
\[
\T_{d(\mu)+d(\alpha)}([g|_{\alpha}|_y])=\T_{d(\mu')+d(\alpha')}([g'|_{\alpha'}|_y]).
\]
But $d(\mu)+d(\alpha)=d(\mu')+d(\alpha')$ as $\mu g\cdot (\alpha)=\mu' g'\cdot (\alpha') $
and $\T_{d(\mu)+d(\alpha)}$ is an automorphism on $Q(\bN^k, G)$. So
\[
[g|_{\alpha}|_y]=[g'|_{\alpha'}|_y].
\]
Since the action is pseudo free, we deduce that
\[
g \vert_{\alpha}=g' \vert_{\alpha'}
\]
by Corollary \ref{C:5.7}.
Thus
\[
Z(\mu,g,\nu) \cap Z(\mu',g',\nu')=\bigcup_{(\alpha,\alpha') \in F} Z(\mu g\cdot \alpha,g|_\alpha, \nu\alpha),
\]
where
\begin{align*}
F=\left\{(\alpha,\alpha') \in \Lambda^{\min}(\nu,\nu'):
\begin{matrix}
(g \cdot \alpha,g'  \cdot \alpha') \in \Lambda^{\min} (\mu,\mu'), \\
g \vert_{\alpha} =g' \vert_{\alpha'}
\end{matrix}
\right\}.
\end{align*}
Therefore, we have shown that $\mathcal{B}_{G,\Lambda}$ is a base of $\G_{G,\Lambda}$.

We now show that $\G_{G,\Lambda}$ is Hausdorff. For this, let us fix two distinct points
$u:=(\mu (g \cdot x);\T_{d(\mu)}([g \vert_{x}]),d(\mu)-d(\nu);\nu x)$ and $u':=(\mu' (g' \cdot  x');\T_{d(\mu')}([g' \vert_{x'}])$, $d(\mu')-d(\nu');\nu' x')$ in $\mathcal{G}_{G,\Lambda}$.

\textit{Case 1:} $\mu (g \cdot x) \neq \mu' (g' \cdot x')$. Then there is $p\ge d(\mu)\lor d(\mu')$ such that
$\mu (g \cdot x) (0,p)\neq \mu' (g' \cdot x') (0,p)$. So
\begin{align*}
&\mu g \cdot x(0,p-d(\mu)) g|_{x(0,p-d(\mu))}\cdot\sigma^{p-d(\mu)}(x) (0,p-d(\mu))\\
&=\mu (g \cdot (x(0,p-d(\mu))\sigma^{p-d(\mu)}(x)) (0,p)\\
&\neq\mu' (g' \cdot (x(0,p-d(\mu'))\sigma^{p-d(\mu')}(x')) (0,p)\\
&=\mu' g' \cdot x'(0,p-d(\mu')) g'|_{x'(0,p-d(\mu'))}\cdot\sigma^{p-d(\mu')}(x') (0,p-d(\mu')).
\end{align*}
Then $Z(\mu (g\cdot x(0,p-d(\mu)),g|_{x(0,p-d(\mu))},\nu x (0,p-d(\mu)))$
and $Z(\mu' (g'\cdot x'(0,p-d(\mu')),g|_{x'(0,p-d(\mu'))},\nu' x'(0,p-d(\mu')))$ separate $u$ and $u'$.

\textit{Case 2:} $d(\mu)-d(\nu) \neq d(\mu')-d(\nu')$. This is directly from the definition of $Z(\mu,g,\nu)$.

\textit{Case 3:} $\nu x \neq \nu' x'$. This is proved similar to Case 1.

\textit{Case 4:}  $\T_{d(\mu)}([g \vert_{x}]) \neq \T_{d(\mu')}([g' \vert_{x'}])$. From Case 1 - Case 3 above, we can assume that
$\mu (g \cdot x)=\mu' (g' \cdot x'), d(\mu)-d(\nu)=d(\mu')-d(\nu')$, and $\nu x= \nu' x'$. Then, as above,
there exist $\alpha \in s(\nu)\Lambda,\alpha' \in s(\nu')\Lambda,y \in s(\alpha)\Lambda^\infty$ such that $(\alpha,\alpha') \in \Lambda^{\min}(\nu,\nu'), (g  \cdot \alpha,g' \cdot \alpha') \in \Lambda^{\min} (\mu,\mu'),x=\alpha y,x'=\alpha' y$, and $g \vert_\alpha \cdot y=g' \vert_{\alpha'} \cdot y$. We claim that $Z(\mu g\cdot\alpha,g|_\alpha, \nu\alpha)$ and $Z(\mu'g'\cdot\alpha',g'|_{\alpha'},\nu'\alpha')$ separate $u$ and $u'$.
To the contrary, let us assume that  there exists $z \in s(\alpha)\Lambda^\infty$ satisfying
$g \vert_\alpha \cdot z=g' \vert_{\alpha'} \cdot z$ and $\T_{d(\mu)}([g \vert_{\alpha z}])=\T_{d(\mu')}([g' \vert_{\alpha' z}])$.
Then, as above, we conclude $g \vert_\alpha=g' \vert_{\alpha'}$ by Corollary \ref{C:5.7}.  Repeatedly applying Lemma \ref{L:translation} yields
\begin{align*}
\T_{d(\mu)}([g \vert_{x}])
&=\T_{d(\mu)}([g \vert_{\alpha y}])=\T_{d(\mu)+d(\alpha)}([g \vert_\alpha|_ y])\\
&=\T_{d(\mu')+d(\alpha')}([g' \vert_{\alpha'}|_ y])=\T_{d(\mu')}([g' \vert_{\alpha' y}])\\
&=\T_{d(\mu')}([g' \vert_{x'}]).
\end{align*}
This is clearly a contradiction.
Therefore $\mathcal{G}_{G,\Lambda}$ is Hausdorff.
\end{proof}


\subsection{$\O_{G,\Lambda}\cong\O_{X_{G,\Lambda}}\cong\mathcal{NO}_{X_{G,\Lambda}}\cong \ca(\G_{G,\Lambda})$}
Let $\Lambda$ be a pseudo free self-similar $k$-graph over a group $G$. Our goal of  this subsection is to show that
$\O_{G,\Lambda}\cong\O_{X_{G,\Lambda}}\cong\mathcal{NO}_{X_{G,\Lambda}}\cong \ca(\G_{G,\Lambda})$
if $|\Lambda^0|<\infty$ and $G$ is amenable.

\begin{lem}
\label{L:propgrpoid}
Let $\Lambda$ be a pseudo free self-similar $k$-graph over a group $G$. Then for $g \in G$ and $\mu,\nu\in \Lambda$, we have
\begin{itemize}
\item[(i)] $\mathds{1}_{Z(\mu,g,\nu)}^*=\mathds{1}_{Z(\nu,g^{-1},\mu)}$; 
\item[(ii)] $\mathds{1}_{Z(\mu,1_G,s(\mu))} \cdot  \mathds{1}_{Z(\nu,1_G,s(\nu))}=\mathds{1}_{Z(\mu\nu,1_G,s(\nu))}$ if $s(\mu)=r(\nu)$;
\item[(iii)] $\mathds{1}_{Z(\mu,1_G,s(\mu))}^*\cdot\mathds{1}_{Z(\mu,1_G,s(\mu))}=\mathds{1}_{Z(s(\mu),1_G,s(\mu))}$;
\item[(iv)] if $s(\mu)=g \cdot s(\nu)$ and $r(\gamma)=s(\nu)$, then
\begin{align*}
\mathds{1}_{Z(\mu,1_G,s(\mu))}\cdot \mathds{1}_{Z(g \cdot \gamma,g \vert_\gamma,\gamma)}\cdot\mathds{1}_{Z(\nu,1_G,s(\nu))}^*=\mathds{1}_{Z(\mu(g \cdot \gamma),g \vert_\gamma,\nu \gamma)};
\end{align*}
\item[(v)] if $s(\mu)=s(\nu)$, then \begin{align*}
\mathds{1}_{Z(v,g,g^{-1} \cdot v)} \cdot \mathds{1}_{Z(\mu,1_G,\nu)}= \begin{cases}
    \mathds{1}_{Z(g \cdot \mu,g \vert_\mu, \nu)} &\text{ if $v=r(g \cdot \mu)$} \\
    0 &\text{ otherwise}.
\end{cases}
\end{align*}
\end{itemize}
\end{lem}
\begin{proof}
The proof is straightforward but tedious, and left to the reader.
\end{proof}

\begin{thm}\label{T:ONC}
Let $\Lambda$ be a pseudo free self-similar $k$-graph over an amenable group $G$ with $\vert\Lambda^0\vert<\infty$.
Then
\[
\O_{G,\Lambda}\cong\O_{X_{G,\Lambda}}\cong\mathcal{NO}_{X_{G,\Lambda}}\cong \ca(\G_{G,\Lambda}).
\]
\end{thm}

\begin{proof}
By Theorem \ref{T:OO}, it remains to show that
$\O_{G,\Lambda}\cong\mathcal{NO}_{X_{G,\Lambda}}\cong \ca(\G_{G,\Lambda})$.
For $\mu \in \Lambda$ and $g\in G$, define
\[
S_\mu:=\mathds{1}_{Z(\mu,1_G,s(\mu))} \text{ and }
U_g:=\sum_{v \in \Lambda^0}\mathds{1}_{Z(v,g,g^{-1} \cdot v)}.
\]
One can check that there exists a homomorphism $\pi:\mathcal{O}_{G,\Lambda} \to \ca(\mathcal{G}_{G,\Lambda})$
such that $\pi(s_\mu)=S_\mu$ and $\pi(u_g)=U_g$ for all $\mu \in \Lambda$ and $g \in G$.
In what follows, we show that $\pi$ is actually surjective. For this, it suffices to show that $\mathds{1}_{Z(\mu, g, \nu)}$ is in the range of $\pi$. In fact, we have
\begin{align*}
S_\mu U_g S_\nu^* =
\begin{cases}
\mathds{1}_{Z(\mu, g, \nu)} & \text{if } s(\mu)=g\cdot s(\nu),\\
0 & \text{otherwise}.
\end{cases}
\end{align*}
To see this, by Lemma \ref{L:propgrpoid} one has
\begin{align*}
&S_\mu U_g S_\nu^*((\alpha(h\cdot x); \T_{d(\alpha)}([h|_x]), d(\alpha)-d(\beta); \beta x))\\
=&U_g S_\nu^*((\mu y; [1], d(\mu); y)^{-1}(\alpha(h\cdot x); \T_{d(\alpha)}([h|_x]), d(\alpha)-d(\beta); \beta x))\\
=&U_g S_\nu^*((y; \T_{d(\alpha)-d(\mu)}([h|_x]), d(\alpha)-d(\beta)-d(\mu); \beta x))\ (\text{if }\mu y = \alpha (h\cdot x))\\
=&\mathds{1}_{Z(s(\nu), 1, \nu)}((z; [g^{-1}|_{g\cdot z}], 0; g\cdot z)(y; \T_{d(\alpha)-d(\mu)}([h|_x]), d(\alpha)-d(\beta)-d(\mu); \beta x))\\
 & (\text{if }\mu y = \alpha (h\cdot x))\\
=&\mathds{1}_{Z(s(\nu), 1, \nu)}((z; [g^{-1}|_{g\cdot z}]\T_{d(\alpha)-d(\mu)}([h|_x]),d(\alpha)-d(\beta)-d(\mu);  \beta x))\\
 & (\text{if } y = g\cdot z)\\
=&1 \Leftrightarrow (\alpha(h\cdot x);\T_{d(\alpha)}([h|_x]), d(\alpha)-d(\beta); \beta x) \in Z(\mu, g, \nu)
\end{align*}
by noticing that $[g^{-1}|_{g\cdot z}]\T_{d(\alpha)-d(\mu)}([h|_x])=[1]\Leftrightarrow\T_{d(\mu)}([g|z])=\T_{d(\alpha)}([h|_x])$.

Let $c:\mathcal{G}_{G,\Lambda} \to \mathbb{Z}^k$ be the ``degree mapping" defined by
\[
c(\mu (g \cdot x);\T_{d(\mu)}([g \vert_{x}]),d(\mu)-d(\nu);\nu x):=d(\mu)-d(\nu).
\]
Then $c$ is a continuous $1$-cocycle. By \cite[Proposition~5.1]{Ren80}, there exists a strongly continuous homomorphism $\alpha: \bT^k \to \Aut(\ca(\mathcal{G}_{G,\Lambda}))$ such that $\alpha_\lambda(f)(x)=\lambda^{c(x)}f(x)$ for all $\lambda \in \bT^k$, $f \in C_c(\mathcal{G}_{G,\Lambda})$ and $x \in \mathcal{G}_{G,\Lambda}$.
Let $\psi$ be the  CNP covariant representation of $X_{G,\Lambda}$ to $\O_{G,\Lambda}$ obtained in the proof of Proposition \ref{P:NOO}.
Observe that $\pi \circ \psi$ is a CNP covariant representation of $X_{G,\Lambda}$ to $\ca(\G_{G,\Lambda})$,
which generates $\ca(\mathcal{G}_{G,\Lambda})$. Also, $\alpha$ is a gauge action for $\pi \circ \psi$.

Since $G$ is amenable, by \cite[Proposition~4.1.9]{BO08} there exists a faithful conditional expectation $E_1:A \to C(\Lambda^0)$ such that
\[
E_1(j(\delta_v))=j(\delta_v),\ E_1(j(\delta_v)j(g))=0 \qforal v \in \Lambda^0,\ 1_G \neq g \in G.
\]
Since $\mathcal{G}_{G,\Lambda}$ is ample, by \cite[Proposition~4.8]{Ren80} there exists a conditional expectation
$E_2:\ca(\mathcal{G}_{G,\Lambda}) \to C_0(\mathcal{G}_{G,\Lambda}^{(0)})$ such that $E_2(f)=f \vert_{\mathcal{G}_{G,\Lambda}^0}$ for all $f \in C_c(\mathcal{G}_{G,\Lambda})$. It is straightforward to see that $\pi \circ \psi_{C(\Lambda^0)}$ is injective. Since $\varphi$ is pseudo free, one has $E_2(U_g)=0$ for $1_G\ne g \in G$. Then
\[
\pi \circ \psi \vert_{C(\Lambda^0)} \circ E_1=E_2 \circ \pi \circ \psi \vert_{A_{G,\Lambda}}.
\]
By \cite[Proposition~3.11]{Kat03}, $\pi \circ \psi$ is injective on $A_{G,\Lambda}$. Therefore by Theorem~\ref{gauge-inv uni thm}, $\pi \circ h$ is an isomorphism. Since $h$ and $\pi$ are both surjective, $h$ and $\pi$ are both isomorphisms.
\end{proof}

The above proof can be summarized pictorially as follows:
\[
\xygraph{
!{<0cm, 0cm>; <1cm,0cm>:<0cm,1cm>::}
!{(0,0)}*+{X_{G,\Lambda}}="a"
!{(3,0)}*+{\O_{{G,\Lambda}}}="b"
!{(3,2)}*+{\N\O_{{X_{G,\Lambda}}}}="c"
!{(3,-2)}*+{\O_{{X_{G,\Lambda}}}}="d"
!{(6,0)}*+{\ca(\G_{G,\Lambda})}="e"
"a":^{\psi}"b" "a":^{j_{cnp}}"c" "a":_{j_{cp}}"d" "d":"e" "c":^{\pi\circ h}"e" "d"-@{^{(}->>}_{h'}"b"
"c"-@{->>}^h"b" "b":^{\pi}"e"
}
\xygraph{
!{<0cm, 0cm>; <1cm,0cm>:<0cm,1cm>::}
!{(0,1)}*+{A_{G,\Lambda}}="a"
!{(3,1)}*+{C(\Lambda^0)}="b"
!{(0,-1)}*+{\ca(\G_{G,\Lambda})}="c"
!{(3,-1)}*+{\ca(\G^{(0)}_{G,\Lambda})}="d"
"c":_{E_2}"d" "a":_{{\pi\circ\psi|_{A_{G,\Lambda}}}}"c" "b":^{\pi\circ\psi|_{C(\Lambda^0)}}"d"
"a":^{E_1}"b"
}
\]

\subsection{The embedding of $\O_\Lambda$ and $G$ into $\O_{G,\Lambda}$}

\begin{prop}\label{embed O_Lambda into O_G,Lambda}
Let $\Lambda$ be a self-similar $k$-graph over a group $G$ with $\vert\Lambda^0\vert<\infty$. Then there exists an injective homomorphism $\iota:\mathcal{O}_\Lambda \to \mathcal{O}_{G,\Lambda}$.
\end{prop}

\begin{proof}
The universal property of $\mathcal{O}_{G,\Lambda}$ induces a strongly continuous homomorphism $\gamma:\bT^k \to \Aut(\mathcal{O}_{G,\Lambda})$ such that $\gamma_z(s_\mu)=z^{d(\mu)}s_\mu$ and $\gamma_z(u_g)=u_g$ for all $z \in \bT^k, \mu \in \Lambda,g \in G$. It yields a gauge action $\alpha:\bT^k \to \Aut(\ca(s_\mu))$ such that $\alpha_z(s_\mu)=z^{d(\mu)}s_\mu$ for all $z \in \bT^k,\mu \in \Lambda$. By Theorem~\ref{gauge-inv uni thm k-graph} and Remark \ref{R:nontrivial}, we obtain an injection $\iota$ from $\ca(\Lambda)$ to $\mathcal{O}_{G,\Lambda}$.
\end{proof}

\begin{prop}
Let $\Lambda$ be a pseudo free self-similar $k$-graph over an amenable group $G$ with $\vert\Lambda^0\vert<\infty$. Then $A_{G,\Lambda}$ embeds into $\O_{G,\Lambda}$. Hence $C^*(G)$ and $G$ also embed into $\O_{G,\Lambda}$.
\end{prop}

\begin{proof}
By the proof of Theorem~\ref{T:ONC}, $\pi\circ\psi|_{A_{G,\Lambda}}$ is an injective homomorphism from $A_{G,\Lambda}$ in $\ca(\mathcal{G}_{G,\Lambda})$. Again by Theorem~\ref{T:ONC}, there exists an embedding $\iota:A_{G,\Lambda} \to \O_{G,\Lambda}$ such that $\iota(j(\delta_v))=s_v$ and $\iota(j(g))=u_g$ for all $v \in \Lambda^0$ and $g \in G$.
\end{proof}

\section{Nuclearity, Simplicity, and a Dichotomy of $\O_{G,\Lambda}$}

\label{S:Pro}

Let $\Lambda$ with $|\Lambda^0|<\infty$ be a pseudo free $k$-graph over an amenable group $G$ .
In this section, we study the properties of $\O_{G,\Lambda}$ in detail.
We prove that $\O_{G,\Lambda}$ is always nuclear and satisfies the UCT; characterize when $\O_{G,\Lambda}$ is simple; and
show that there is a dichotomy for $\O_{G,\Lambda}$ if it is simple: it is either stably finite or purely infinite.
In particular, when $\O_{G,\Lambda}$ is simple and purely infinite, then
it is classifiable by its K-theory thanks to the Kirchberg-Phillips classification theorem \cite{Kir94, Phi00}.

\subsection{Nuclearity and simplicity of $\O_{G,\Lambda}$}

\begin{defn}
Let $\Lambda$ be a self-similar $k$-graph over a group $G$.
Then $\Lambda$ is said to be \emph{$G$-cofinal} if for any $x \in \Lambda^\infty$ and $v \in \Lambda^0$, there exist $p \in \mathbb{N}^k, \mu \in \Lambda,g \in G$ such that $s(\mu)=x(p,p)$ and $g \cdot v=r(\mu)$.
\end{defn}

Clearly, if $\Lambda$ is cofinal in the usual sense, then it is $G$-cofinal.
We give a characterization of the $G$-cofinality in terms of finite paths, which is a generalization of \cite[Proposition~A.2]{LS10}.

\begin{prop}\label{characterization of G-cof}
Let $\Lambda$ be a self-similar $k$-graph over a group $G$. Then $\Lambda$ is $G$-cofinal if and only if for any $v,w \in \Lambda^0$, there exists $p \in \mathbb{N}^k$ such that $(G \cdot v)\Lambda s(\mu) \neq \mt$ for all $\mu \in w\Lambda^p$.
\end{prop}

\begin{proof}
First of all, suppose that $\varphi$ is $G$-cofinal. To the contrary, let us assume that there exist $v,w \in \Lambda^0$ such that for any $p \in \mathbb{N}^k$, there exists $\mu \in w\Lambda^p$ satisfying $(G \cdot v)\Lambda s(\mu) = \mt$.
If $k<\infty$, then there exists a sequence $\{\mu_n\}_{n=1}^{\infty} \subset w\Lambda$ such that for any $n \geq 1$,
\begin{itemize}
\item $d(\mu_n)=n(e_1+\cdots +e_k)$;
\item $\mu_{n+1}=\mu_n \nu$ for some $\nu \in \Lambda$; and
\item for any $p \in \mathbb{N}^k$, there exists $\lambda \in s(\mu_n)\Lambda^p$ such that $(G \cdot v)\Lambda s(\lambda) = \mt$.
\end{itemize}
If $k=\infty$, then there exists a sequence $\{\mu_n\}_{n=1}^{\infty} \subset w\Lambda$ such that for any $n \geq 1$,
\begin{itemize}
\item $d(\mu_n)=n(e_1+\cdots +e_n)$;
\item $\mu_{n+1}=\mu_n \nu$ for some $\nu \in \Lambda$; and
\item for any $p \in \mathbb{N}^k$, there exists $\lambda \in s(\mu_n)\Lambda^p$ such that $(G \cdot v)\Lambda s(\lambda) = \mt$.
\end{itemize}
In either case, we obtain $x \in \Lambda^\infty$ satisfying that for any $p \in \mathbb{N}^k, (G \cdot v)\Lambda x(p,p)= \mt$, which contradicts with the $G$-cofinality condition.


Conversely, suppose that for any $v,w \in \Lambda^0$, there exists $p \in \mathbb{N}^k$ such that $(G \cdot v)\Lambda s(\mu) \neq \mt$ for all $\mu \in w\Lambda^p$. Fix $v \in \Lambda^0$ and $x \in \Lambda^\infty$. By the assumption, there exists $p \in \mathbb{N}^k$ such that $(G \cdot v)\Lambda s(\mu) \neq \mt$ for all $\mu \in x(0,0)\Lambda^p$. So $(G \cdot v) \Lambda x(p,p)\neq \mt$. Hence $\varphi$ is $G$-cofinal.
\end{proof}

\begin{prop}\label{minimal iff cofinal}
Let $\Lambda$ be a pseudo free self-similar $k$-graph over a group $G$. Then $\mathcal{G}_{G,\Lambda}$ is minimal if and only if $\Lambda$ is $G$-cofinal.
\end{prop}

\begin{proof}
Suppose that $\mathcal{G}_{G,\Lambda}$ is minimal. Fix $x \in \Lambda^\infty$ and $v \in \Lambda^0$. Let $F:=\overline{\bigcup_{g \in G}r(s^{-1}(g \cdot x))}$, which is a nonempty closed invariant subset.
Since $\mathcal{G}_{G,\Lambda}$ is minimal, $F=\mathcal{G}_{G,\Lambda}^0$. Then $Z(v) \subset F$. So there exist $g \in G, \gamma \in s^{-1}(g \cdot x)$ such that $r(\gamma) \in Z(v)$.
Let us assume that $\gamma=(\mu(\tilde g \cdot y), \Phi_{d(\mu)}([\tilde g|_y]), d(\mu)-d(\nu), \nu y)$, where $\nu y=g\cdot x$ and $s(\mu)=\tilde g s(\nu)$.
So
\[
s(\mu)=\tilde g s(\nu)=\tilde g r(y)=\tilde g (s(g\cdot x(0,d(\nu)))=\tilde g g x(d(\nu), d(\nu)).
\]
Thus
\[
s((\tilde g g)^{-1}\mu)=x(d(\nu), d(\nu))\text{ and } r((\tilde g g)^{-1}\mu)=(\tilde g g)^{-1} v.
\]
Therefore, $\varphi$ is $G$-cofinal.

Conversely, suppose that $\varphi$ is $G$-cofinal. Fix a nonempty open invariant subset $U$. Then there exists $\mu \in \Lambda$ such that
$Z(\mu) \subset U$. For $x \in \Lambda^\infty$, since $\varphi$ is $G$-cofinal, there exist
$p \in \mathbb{N}^k, \nu \in \Lambda,g \in G$ such that $s(\nu)=x(p,p), g \cdot s(\mu)=r(\nu)$. Let
\begin{align*}
\gamma:&=(x(0,p)(g^{-1}|_\nu)^{-1} (g^{-1}|_\nu\cdot \sigma^p(x));[\T_p(g \vert_{g^{-1} \cdot \nu})],p-d(\mu(g^{-1}\cdot \nu));\\
 &\hskip 8cm \mu g^{-1} \cdot \nu (g^{-1}|_\nu\cdot \sigma^p(x)))\\
&=(x;[\T_p(g \vert_{g^{-1} \cdot \nu})],p-d(\mu(g^{-1}\cdot \nu));\mu (g^{-1} \cdot (\nu \sigma^p(x)))).
\end{align*}
Note that $\gamma\in \G_{G,\Lambda}$ as
\begin{align*}
x(p,p)
&=s(x(0,p))=s(\nu)=(g^{-1}|_\nu)^{-1}s(\mu g^{-1}\cdot \nu)\\
&=(g^{-1}|_\nu)^{-1}r(g^{-1}|_\nu\sigma^p(x))=r(\sigma^p(x)).
\end{align*}
Then $s(\gamma) \in U$. So $r(\gamma)=x \in U$ as $U$ is invariant. Since $x$ is arbitrary, $U=\Lambda^\infty$. Hence $\mathcal{G}_{G,\Lambda}$ is minimal.
\end{proof}

\begin{defn}
Let $G$ be a self-similar group over $\Lambda$.
Then $\Lambda$ is said to be \emph{$G$-aperiodic} if for any $v \in \Lambda^0$ there exists $x \in v\Lambda^\infty$ such that  for $g \in G,p,q \in \mathbb{N}^k$, if $g \neq 1_G$ or $p \neq q$ then $\sigma^{p}(x) \neq g \cdot \sigma^q(x)$.
\end{defn}

It is easy to see that the $G$-aperiodicity of $\Lambda$ implies the aperiodicity of $\Lambda$.

\begin{prop}\label{top principal iff ape}
Let $\Lambda$ be a pseudo free self-similar $k$-graph over a group $G$.
Then $\mathcal{G}_{G,\Lambda}$ is topologically principal if and only if $\Lambda$ is $G$-aperiodic.
\end{prop}

\begin{proof}
Firstly suppose that $\mathcal{G}_{G,\Lambda}$ is topologically principal. Suppose that $\varphi$ not is $G$-aperiodic, for a contradiction. Then there exists $v \in \Lambda^0$ such that for any $x \in v \Lambda^\infty$ there exist $g \in G, p,q \in \mathbb{N}^k$ with $g \neq 1_G,\sigma^{p}(x)= g \cdot \sigma^q(x)$ or $p \neq q,\sigma^{p}(x)= g \cdot \sigma^q(x)$. Fix such $x \in v \Lambda^\infty$.

Case $1$. There exist $g \in G, p,q \in \mathbb{N}^k$ such that $p \neq q, \sigma^{p}(x)= g \cdot \sigma^q(x)$. Since
$p \neq q$, clearly $(x;[\T_p(g \vert_{\sigma^q(x)})],p-q;x) \neq (x;1_{Q(\mathbb{N}^k,G)},0;x)$.

Case $2$. There exist $g \in G, p,q \in \mathbb{N}^k$ such that $g \neq 1_G$ and $\sigma^{p}(x)= g \cdot \sigma^q(x)$. We may assume that $p=q$ by Case 1.
Since $\varphi$ is pseudo free, $[\T_p(g \vert_{\sigma^q(x)})] \neq 1_{Q(\mathbb{N}^k,G)}$ by Corollary \ref{C:5.7}. So
$(x;[\T_p(g \vert_{\sigma^q(x)})],0;x) \neq (x;1_{Q(\mathbb{N}^k,G)},0;x)$.
This yields $Z(v) \cap \{y \in \Lambda^\infty:\text{the isotropy at $y$ is trivial}\}=\mt$, which contradicts the assumption that $\mathcal{G}_{G,\Lambda}$ is topologically principal. Hence $\Lambda$ is $G$-aperiodic.

Conversely, suppose that $\Lambda$ is $G$-aperiodic. If $\mathcal{G}_{G,\Lambda}$ were not topologically principal, then there would exist
$\mu \in \Lambda$ such that $Z(\mu) \subset \Lambda^\infty \setminus \overline{\{y \in \Lambda^\infty:\text{the isotropy at $y$ is trivial}\}}$. Fix $x \in \mu \Lambda^\infty$. There exist $\alpha,\beta \in \Lambda,g \in G, y \in \Lambda^\infty$ with $s(\alpha)=r(g \cdot y),s(\beta)=r(y), \alpha (g \cdot y)=\beta y=x$ such that $[\T_{d(\alpha)}(g \vert_y)] \neq 1_{Q(\mathbb{N}^k,G)}$, implying  $g \neq 1_G$ or $d(\alpha) \neq d(\beta)$.
From $\alpha (g \cdot y)=\beta y=x$, it follows that $y=\sigma^{d(\beta)}(x)$ and $g\cdot y=\sigma^{d(\alpha)}(x)$. So
$\sigma^{d(\beta)}(x)=g^{-1}\sigma^{d(\alpha)}(x)$. This contradicts the $G$-aperiodicity of $\Lambda$.
Therefore $\mathcal{G}_{G,\Lambda}$ is topologically principal.
\end{proof}

\begin{thm}
\label{T:AmeSim}
Let $\Lambda$ be a pseudo free self-similar $k$-graph over an amenable group $G$ with $\vert\Lambda^0\vert<\infty$. Then
\begin{itemize}
\item[(i)]
\label{O_G,Lambda is ame} $\mathcal{O}_{G,\Lambda}$ is nuclear and satisfies the UCT;
\item[(ii)]
\label{O_G,Lambda is simple} $\mathcal{O}_{G,\Lambda}$ is simple if and only if $\Lambda$ is $G$-cofinal and $G$-aperiodic.
\end{itemize}
\end{thm}

\begin{proof}
(i) Since $G$ is amenable, $A_{G,\Lambda}$ is amenable. By \cite[Theorem~3.21]{AM15}, $\mathcal{O}_{X_{G,\Lambda}}$ is amenable.
It follows from Theorem~\ref{T:ONC} that both $\mathcal{O}_{G,\Lambda}$ and  $\ca(\mathcal{G}_{G,\Lambda})$ are amenable.
By \cite[Corollary~5.6.17, Theorem~5.6.18, Corollary 9.4.4]{BO08}, $\mathcal{G}_{G,\Lambda}$ is amenable and $\ca(\mathcal{G}_{G,\Lambda}) \cong \ca_r(\mathcal{G}_{G,\Lambda})$.
Since $\mathcal{G}_{G,\Lambda}$ is amenable, $\ca(\mathcal{G}_{G,\Lambda})$ satisfies the UCT (\cite{Tu99}), and so does $\mathcal{O}_{G,\Lambda}$ by Theorem~\ref{T:ONC}
again.

(ii) This directly follows from \cite[Theorem~5.1]{BCFS14}, Theorem~\ref{T:ONC}, Proposition~\ref{minimal iff cofinal}, and Proposition~\ref{top principal iff ape}.
\end{proof}

From the above proof, we get the following by-product.

\begin{prop}
\label{P:CCr}
Let $\Lambda$ be a pseudo free self-similar $k$-graph over an amenable group $G$ with $\vert\Lambda^0\vert<\infty$. Then $\G_{G,\Lambda}$ is amenable and $\ca(\mathcal{G}_{G,\Lambda}) \cong \ca_r(\mathcal{G}_{G,\Lambda})$.
\end{prop}

\begin{rem}
Theorem \ref{T:AmeSim} and Proposition \ref{P:CCr} generalize \cite[Theorem 16.1, Corollary 10.16, Corollary 10.18]{EP17} when the action is pseudo free.
\end{rem}

\subsection{A dichotomy for $\O_{G,\Lambda}$}

The following four lemmas are inspired by some results in \cite{PSS17}.

\begin{lem}\label{[chi_Z(mu)]=[chi_Z(nu)]}
Let $\Lambda$ be a pseudo free self-similar $k$-graph over a group $G$ with $\vert\Lambda^0\vert<\infty$.
Then for $\mu,\nu \in \Lambda$ with $s(\mu)=s(\nu)$, we have $[\mathds{1}_{Z(\mu)}]=[\mathds{1}_{Z(\nu)}]$ in $S(\mathcal{G}_{G,\Lambda})$.
\end{lem}
\begin{proof}
This follows from $\mathds{1}_{Z(\mu)}=\mathds{1}_{r(Z(\mu,1_G,\nu))}$ and $\mathds{1}_{Z(\nu)}=\mathds{1}_{s(Z(\mu,1_G,\nu))}$.
\end{proof}

\begin{lem}\label{[chi_Z(w)] leq [chi_Z(v)]}
Let $\Lambda$ be a pseudo free self-similar $k$-graph over a group $G$ with $\vert\Lambda^0\vert<\infty$.
Then for $v,w \in \Lambda^0$, for $g \in G$, we have $(g \cdot v)\Lambda w \neq \mt \implies [\mathds{1}_{Z(w)}] \leq [\mathds{1}_{Z(v)}]$. Hence if $\Lambda$ is $G$-strongly connected then $[\mathds{1}_{Z(w)}] \leq [\mathds{1}_{Z(v)}]$ for all $v,w \in \Lambda^0$.
\end{lem}

\begin{proof}
Fix $\mu \in (g \cdot v)\Lambda w$. Then $Z(g^{-1} \cdot \mu, g^{-1}, w)$ is a compact open subset of $\mathcal{G}_{G,\Lambda}$. Observe that
\begin{align*}
\mathds{1}_{Z(v)}&=\mathds{1}_{Z(g^{-1} \cdot \mu)}+\sum_{\nu \in v \Lambda^{d(\mu)} \setminus \{g^{-1} \cdot \mu\}}\mathds{1}_{Z(\nu)}
\\&=\mathds{1}_{r(Z(g^{-1} \cdot \mu, g^{-1}, w))}+\sum_{\nu \in v \Lambda^{d(\mu)} \setminus \{g^{-1} \cdot \mu\}}\mathds{1}_{Z(\nu)},
\end{align*}
and
\[
\mathds{1}_{s(Z(g^{-1} \cdot \mu, g^{-1}, w))}+\sum_{\nu \in v \Lambda^{d(\mu)} \setminus \{g^{-1} \cdot \mu\}}\mathds{1}_{Z(\nu)}=
\mathds{1}_{Z(w)}+\sum_{\nu \in v \Lambda^{d(\mu)} \setminus \{g^{-1} \cdot \mu\}}\mathds{1}_{Z(\nu)}.
\]
So by Lemma \ref{[chi_Z(mu)]=[chi_Z(nu)]}
\begin{align*}
&[\mathds{1}_{Z(w)}]+\big[\sum_{\nu \in v \Lambda^{d(\mu)} \setminus \{g^{-1} \cdot \mu\}}\mathds{1}_{Z(\nu)}\big]\\
&=[\mathds{1}_{r(Z(g^{-1} \cdot \mu, g^{-1}, w))}]+\big[\sum_{\nu \in v \Lambda^{d(\mu)} \setminus \{g^{-1} \cdot \mu\}}\mathds{1}_{Z(\nu)}\big]\\
&=[\mathds{1}_{s(Z(g^{-1} \cdot \mu, g^{-1}, w))}]+\big[\sum_{\nu \in v \Lambda^{d(\mu)} \setminus \{g^{-1} \cdot \mu\}}\mathds{1}_{Z(\nu)}\big]\\
&=[\mathds{1}_{Z(v)}].
\end{align*}
Thus $[\mathds{1}_{Z(w)}] \leq [\mathds{1}_{Z(v)}]$.
\end{proof}

\begin{lem}\label{S(g_{G,Lambda}) is purely inf if vLambda^p>1}
Let $\Lambda$ be a pseudo free self-similar $k$-graph over a group $G$ with $\vert\Lambda^0\vert<\infty$. Suppose that $\Lambda$ is $G$-strongly connected and that there exist $v \in \Lambda^0$ and $p \in \mathbb{N}^k$ such that $\vert v\Lambda^p\vert>1$. Then $S(\mathcal{G}_{G,\Lambda})$ is purely infinite.
\end{lem}

\begin{proof}
It suffices to show that for any $\mu \in \Lambda, 2[\mathds{1}_{Z(\mu)}] \leq [\mathds{1}_{Z(\mu)}]$. Take $\mu_1 \neq \mu_2 \in v\Lambda^p$. Then for $\mu \in \Lambda$, due to Lemmas~\ref{[chi_Z(mu)]=[chi_Z(nu)]} and \ref{[chi_Z(w)] leq [chi_Z(v)]}, we have
\begin{align*}
2[\mathds{1}_{Z(\mu)}]
&=2[\mathds{1}_{Z(s(\mu))}]\ (\text{by Lemma } \ref{[chi_Z(mu)]=[chi_Z(nu)]})\\
&\leq [\mathds{1}_{Z(s(\mu_1))}]+[\mathds{1}_{Z(s(\mu_2))}] \ (\text{by } G\text{-strong connectedness})\\
&=[\mathds{1}_{Z(\mu_1)}]+[\mathds{1}_{Z(\mu_2)}]\ (\text{by Lemma } \ref{[chi_Z(mu)]=[chi_Z(nu)]})\\
&\leq [\mathds{1}_{Z(v)}] \ (\text{as }r(\mu_1)=r(\mu_2)=v)\\
&\leq [\mathds{1}_{Z(s(\mu))}]\ (\text{by Lemma }\ref{[chi_Z(w)] leq [chi_Z(v)]})\\
&=[\mathds{1}_{Z(\mu)}] \ (\text{by Lemma } \ref{[chi_Z(mu)]=[chi_Z(nu)]}).
\end{align*}
So $S(\mathcal{G}_{G,\Lambda})$ is purely infinite.
\end{proof}

\begin{lem}\label{S(g_G,H Lambda)iso to S(g_G,Lambda)}
Let $\Lambda$ be a pseudo free self-similar $k$-graph over a group $G$ with $\vert\Lambda^0\vert<\infty$. Suppose that $\Lambda$ is $G$-cofinal.
Then  $S(\mathcal{G}_{G,H \Lambda})$ is isomorphic to $S(\mathcal{G}_{G,\Lambda})$ for any nonempty $G$-hereditary subset $H \subseteq \Lambda^0$.
\end{lem}

\begin{proof}
Notice that $(H \Lambda)^\infty$ is a compact open subset of $\Lambda^\infty$. For $f \in C((H \Lambda)^\infty,\mathbb{N})$, denote by $\widetilde{f}$ the extension of $f$ to $\Lambda^\infty$, which is zero outside $(H \Lambda)^\infty$. It is straightforward to see that there exists an injective homomorphism $\pi:S(\mathcal{G}_{G,H \Lambda}) \to S(\mathcal{G}_{G,\Lambda})$ such that $\pi([f])=[\widetilde{f}]$ for all $f \in C((H \Lambda)^\infty,\mathbb{N})$. In order to prove that $\pi$ is surjective, we only need to show that for $\mu \in \Lambda, [\mathds{1}_{Z(\mu)}]$ is in the image of $\pi$. Pick an arbitrary $v \in H$. Since $\varphi$ is $G$-cofinal, by Proposition~\ref{characterization of G-cof}, there exists $p \in \mathbb{N}^k$, such that $(G \cdot v) \Lambda s(\nu) \neq \mt$ for all $\nu \in s(\mu)\Lambda^p$. Write $s(\mu) \Lambda^p=\{\nu_i\}_{i=1}^{n}$. Then for $1 \leq i \leq n$, let $g_i \in G,\alpha_i \in \Lambda$ such that $\alpha_i \in (g_i \cdot v) \Lambda s(\nu_i) (\subseteq H \Lambda)$. By Lemma~\ref{[chi_Z(mu)]=[chi_Z(nu)]}, we have
\begin{align*}
\big[\mathds{1}_{Z(\mu)}\big]=\sum_{i=1}^{n}\big[\mathds{1}_{Z(\mu \nu_i)}\big]=\sum_{i=1}^{n}\big[\mathds{1}_{Z(\alpha_i)}\big]=\sum_{i=1}^{n}\pi\big([\mathds{1}_{Z(\alpha_i)}]\big),
\end{align*}
where Lemma \ref{[chi_Z(mu)]=[chi_Z(nu)]} is used the above second ``=" as $s(\mu\nu_i)=s(\alpha_i)$.
Therefore $\pi$ is surjective.
\end{proof}

\begin{thm}\label{T:dicho}
Let $\Lambda$ be a pseudo free self-similar $k$-graph over an amenable group $G$ with $\vert\Lambda^0\vert<\infty$.
Suppose that $\mathcal{O}_{G,\Lambda}$ is simple. Then $\O_{G,\Lambda}$ is either stably finite or purely infinite.
More precisely,
\begin{itemize}
\item[(i)]
$\mathcal{O}_{G,\Lambda}$ is stably finite if and only if $\Lambda$ has nonzero graph traces;
\item[(ii)]
$\mathcal{O}_{G,\Lambda}$ is purely infinite if and only if $\Lambda$ has no nonzero graph traces.
\end{itemize}
\end{thm}

\begin{proof}
By Theorem~\ref{T:ONC}, it is equivalent to prove that $\ca(\G_{G,\Lambda})$ has the given dichotomy when it is simple.
But if $\ca(\G_{G,\Lambda})$ is simple, it then follows from Propositions~\ref{minimal iff cofinal} and \ref{top principal iff ape}
that $\G_{G,\Lambda}$ satisfies the conditions required in Theorem \ref{T:BLRS}.

(i) We first suppose that $\ca(\G_{G,\Lambda})$ is stably finite. By Theorem~\ref{T:BLRS}, there exists a faithful tracial state $\tau$ on $\mathcal{O}_{G,\Lambda}$. Define $gt:\Lambda^0 \to [0,\infty)$ by $gt(v):=\tau(s_v)$. It is straightforward to see that $gt$ is a faithful (in particular, nonzero) graph trace on $\Lambda$.

Conversely, suppose that there exists a nonzero graph trace $gt$ on $\Lambda$. By \cite[II.4.8]{Ren80} and Proposition \ref{P:CCr}, there exists a faithful conditional expectation
$E_2:\ca(\mathcal{G}_{G,\Lambda}) \to C_0(\mathcal{G}_{G,\Lambda}^{(0)})$ such that $E_2(f)=f \vert_{\mathcal{G}_{G,\Lambda}^{(0)}}$ for all $f \in C_c(\mathcal{G}_{G,\Lambda})$. Since there is an isomorphism $\pi$ from $\mathcal{O}_{G,\Lambda}$ onto $\ca(\mathcal{G}_{G,\Lambda})$ by Theorem \ref{T:ONC}, $E_2\circ\pi=:E$ gives a faithful conditional expectation from $\mathcal{O}_{G,\Lambda}$ onto $\overline{\mathrm{span}}\{s_\mu s_\mu^*:\mu \in \Lambda\}$ such that,
for $\mu,\nu \in \Lambda,g \in G$ with $s(\mu)=g \cdot s(\nu)$, one has
\begin{align*}
E(s_\mu u_g s_\nu^*)= \begin{cases}
    s_\mu s_\mu^*  &\text{ if $\mu=\nu,g=1_G$} \\
    0 &\text{ otherwise }.
\end{cases}
\end{align*}
Also, there exists a bounded linear functional $\phi:\overline{\mathrm{span}}\{s_\mu s_\mu^*:\mu \in \Lambda\} \to \mathbb{C}$
satisfying $\phi(s_\mu s_\mu^*)=gt(s(\mu))$. Define $\tau:=\phi \circ E$. It is straightforward to see that $\tau$ is a nonzero trace. Normalizing $\tau$ yields a faithful tracial state on $\mathcal{O}_{G,\Lambda}$ (the faithfulness follows from the simplicity of $\mathcal{O}_{G,\Lambda}$). Therefore by Theorem~\ref{T:BLRS} $\mathcal{O}_{G,\Lambda}$ is stably finite.

(ii) By Theorem~\ref{T:BLRS}, it suffices to show that $S(\mathcal{G}_{G,\Lambda})$ is purely infinite.
Since $\Lambda^0$ is finite, as in \cite{PSS17}, we can find a cycle $\mu$ with $d(\mu)\ge (1,\ldots, 1)\in\bN^k$.
Let $x=\mu^\infty\in\Lambda^\infty$, and $v:=x(0,0)$.
Define $H:=\{w \in \Lambda^0:(G \cdot v) \Lambda w \neq \mt\}$, which is a nonempty $G$-hereditary subset of $\Lambda^0$. By Lemma~\ref{S(g_G,H Lambda)iso to S(g_G,Lambda)}, we only need to show that $S(\mathcal{G}_{G,H \Lambda})$ is purely infinite.
Since $\ca(\G_{G,\Lambda})$ is simple, $\Lambda$ is $G$-cofinal and $G$-aperiodic by Theorems \ref{T:ONC} and \ref{T:AmeSim}.
It is straightforward to see that $H \Lambda$ is also $G$-cofinal and $G$-aperiodic. So there exist $v_0 \in H$ and $p \in \mathbb{N}^k$ such that $\vert v_0\Lambda^p\vert>1$. (Otherwise, $H\Lambda$ is periodic, implying $G$-periodic). Fix $w \in H$. Since $H \Lambda$ is $G$-cofinal, there exists $p \in \mathbb{N}^k$ such that $(G \cdot w)\Lambda x(p,p) \neq \mt$. By the property of $x$, we deduce that $(G \cdot w)\Lambda v \neq \mt$. Thus, $H \Lambda$ is $G$-strongly connected. By Lemma~\ref{S(g_{G,Lambda}) is purely inf if vLambda^p>1}, $S(\mathcal{G}_{G,H \Lambda})$ is purely infinite. Therefore we are done.
\end{proof}

\begin{rem}
Theorem \ref{T:dicho} generalizes \cite[Corollary 16.3]{EP17}, which says that a simple self-similar (directed) graph C*-algebra is always pure infinite, and
so there is no dichotomy in the rank-1 case. In fact, by Theorem \ref{T:dicho}, it is sufficient to show that, for a 1-graph $\Lambda$, $\Lambda$ has no nonzero graph traces
if $\O_{G,\Lambda}$ is simple. To see this, notice that if $\O_{G,\Lambda}$ is simple, then $\Lambda$ is $G$-aperiodic, and so aperiodic,  by Theorem \ref{T:AmeSim}.  Thus
every cycle of $\Lambda$ has an entry. This forces that $\Lambda$ has no nonzero graph traces (also cf. \cite{PR06}).
\end{rem}

Combining Theorem \ref{T:dicho} and \cite[Theorem~1.1]{CHS16}, one can get more explicit characterizations for stable finiteness and pure infiniteness of $O_{G,\Lambda}$. For this, we must introduce some new notation
(see \cite{CHS16}).
Let $\Lambda$ be a $k$-graph. Denote by $\mathbb{N}^{\Lambda^0}$ the set of functions from $\Lambda^0$ to $\mathbb{N}$ with the finite support, which is a unital abelian semigroup,
and $M_{\Lambda^0}(\mathbb{N})$ the set of functions from $\Lambda^0 \times \Lambda^0$ to $\mathbb{N}$. For $T \in M_{\Lambda^0}(\mathbb{N})$, we regard $T$ as an endomorphism of $\mathbb{Z}^{\Lambda^0}$ such that $(Tx)(v)=\sum_{w \in \Lambda^0}T(v,w)x(w)$ for all $x \in \mathbb{Z}^{\Lambda^0},v \in \Lambda^0$. For $p \in \mathbb{N}^k,v,w \in \Lambda^0$, define $T_p(v,w):=\vert v\Lambda^p w\vert$ and $T_p^t(v,w):=\vert w\Lambda^p v\vert$.

\begin{cor}
Let  $G$ be an amenable group, and $\Lambda$ be a pseudo free self-similar $k$-graph over $G$, which is $G$-cofinal and $G$-aperiodic.
Then $\mathcal{O}_{G,\Lambda}$ is stably finite if and only if $(\sum_{i \in F}\mathrm{Im}(1-T_{e_i}^t)) \cap \mathbb{N}^{\Lambda^0}=\{0\}$
for any finite subset $F \subseteq \{1,\dots,k\}$, and $\mathcal{O}_{G,\Lambda}$ is purely infinite if and only if for any finite subset
$\sum_{i \in F}\mathrm{Im}(1-T_{e_i}^t)) \cap \mathbb{N}^{\Lambda^0}\neq\{0\}$ for any finite subset $F \subseteq \{1,\dots,k\}$.
\end{cor}

As an immediate consequence of Theorems \ref{T:AmeSim} and \ref{T:dicho}, we have

\begin{cor}
Let $\Lambda$ be a pseudo free self-similar $k$-graph over an amenable group $G$ with $\vert\Lambda^0\vert<\infty$.
Suppose that $\Lambda$ is $G$-aperiodic, $G$-cofinal, and has no nonzero graph traces. Then $\O_{G,\Lambda}$
is a Kirchberg algebra.
\end{cor}

There are a lot of interesting topics to continue the study of this paper. In a forthcoming  paper, we will investigate
the KMS states of $\O_{G,\Lambda}$ (the recent work \cite{CS17, LRRW17} is related to the rank-1 case).

\subsection*{Acknowledgments}

The first author wants to thank the second author for being generous in sharing ideas and for numerous valuable discussions.

\end{document}